\documentclass[12pt,reqno]{amsart}
\usepackage{color}
\usepackage{amsmath, amsthm, amssymb}
\usepackage{amsfonts}
\usepackage{mathrsfs}
\usepackage{amscd}
\usepackage{wrapfig}
\usepackage{graphicx}
\usepackage[active]{srcltx}

\usepackage[ansinew]{inputenc}
\usepackage[dvips]{epsfig}
\usepackage{graphicx}
\usepackage[english]{babel}
\usepackage{hyperref}

\theoremstyle{plain}
\newtheorem{thm}{Theorem}[section]
\newtheorem{cor}[thm]{Corollary}
\newtheorem{lem}[thm]{Lemma}
\newtheorem{prop}[thm]{Proposition}

\theoremstyle{definition}
\newtheorem{defi}[thm]{Definition}

\theoremstyle{remark}

\numberwithin{equation}{section}

\newcommand{\average}{{\mathchoice {\kern1ex\vcenter{\hrule height.4pt
width 6pt depth0pt} \kern-9.7pt} {\kern1ex\vcenter{\hrule
height.4pt width 4.3pt depth0pt} \kern-7pt} {} {} }}

\def\R{\mathbb{R}}

\newcommand{\Q}{{\mathcal Q}}
\newcommand{\B}{{\mathcal B}}

\begin{document}

\title[Free boundary in the parabolic fractional obstacle problem]{Free boundary regularity in the\\ parabolic fractional obstacle problem}

\author[B. Barrios]{Bego\~na Barrios}
\address{Departamento de An\'{a}lisis Matem\'{a}tico, Universidad de La Laguna
\hfill \break \indent C/Astrof\'{\i}sico Francisco S\'{a}nchez s/n, 38271 -- La Laguna, Spain}
\email{bbarrios@ull.es}

\author[A. Figalli]{Alessio Figalli}
\address{The University of Texas at Austin, Department of Mathematics, 2515 Speedway, Austin, TX 78751, USA}
\email{figalli@math.utexas.edu}

\author[X. Ros-Oton]{Xavier Ros-Oton}
\address{The University of Texas at Austin, Department of Mathematics, 2515 Speedway, Austin, TX 78751, USA}
\email{ros.oton@math.utexas.edu}

\thanks{The first author was partially supported by a MEC-Juan de la Cierva postdoctoral fellowship
 (Spain) and MTM2013-40846-P, MINECO.}



\maketitle

\begin{abstract}
The parabolic obstacle problem for the fractional Laplacian naturally arises in American option models when the assets prices are driven by pure jump L\'evy processes.
In this paper we study the regularity of the free boundary.
Our main result establishes that, when $s>\frac12$, the free boundary is a $C^{1,\alpha}$ graph in $x$ and $t$ near any regular free boundary point $(x_0,t_0)\in \partial\{u>\varphi\}$.
Furthermore, we also prove that solutions $u$ are $C^{1+s}$ in $x$ and $t$ near such points, with a precise expansion of the form
\[u(x,t)-\varphi(x)=c_0\bigl((x-x_0)\cdot e+a(t-t_0)\bigr)_+^{1+s}+o\bigl(|x-x_0|^{1+s+\alpha}+ |t-t_0|^{1+s+\alpha}\bigr),\]
with $c_0>0$, $e\in \mathbb{S}^{n-1}$, and $a>0$.
\end{abstract}

\section{Introduction}

Obstacle problems of the form
\begin{equation}\label{intro1}
\min\bigl\{\mathcal L v,\,v-\varphi\bigr\}=0\quad\textrm{in}\quad \R^n\times (0,T),\quad
\end{equation}
\begin{equation}\label{intro2}
\ v(T)=\varphi\quad\textrm{in}\quad \R^n,
\end{equation}
arise in the study of optimal stopping problems for stochastic processes.
When the underlying stochastic process is a pure-jump L\'evy process, then $\mathcal L$ is a (backward) parabolic integro-differential operator of the form
\[\mathcal Lv(x,\tau)=-\partial_\tau v-\int_{\R^n}\bigl(v(x+z,\tau)-v(x,\tau)-\nabla v(x,\tau)\cdot z \chi_{B_1}(z)\bigr)\mu(dz),\]
where $\mu$ is the L\'evy measure (or jump measure).

An important motivation for studying such problems comes from mathematical finance \cite{Merton}, where this type of obstacle problems are used to model rational prices of American options.
In that context, the obstacle $\varphi$ is a payoff function, $T$ is the expiration date of the option, and the set $\{v=\varphi\}$ is called the exercise region; see the book \cite{CT} for a description of the model.

Here we assume that the underlying L\'evy process is stable (i.e., scale invariant) and rotationally symmetric.
Then, after the change of variable $t=T-\tau$, problem \eqref{intro1}-\eqref{intro2} becomes
\begin{equation}\label{obst-pb}\begin{split}
\min\bigl\{\partial_t u+(-\Delta)^su,\,u-\varphi\bigr\}&=0\quad\textrm{in}\ \, \R^n\times(0,T],\,\\
u(\cdot,0)&=\varphi\,\,\,\,\, \textrm{in}\ \, \R^n,
\end{split}\end{equation}
where $\varphi:\R^n\to\R$ is a smooth obstacle, and
\[\qquad (-\Delta)^sw(x)=c_{n,s}\,\textrm{p.v.}\int_{\R^n}(w(x)-w(x+z))\frac{dz}{|z|^{n+2s}},\qquad s\in(0,1).\]

Note that the scaling of the parabolic equation $\partial_t u+(-\Delta)^su=0$ changes completely 
depending on the value of $s$: while for $s>1/2$ space scales slower than time (as in the case of the classical heat equation $s=1$), for $s=1/2$ the scaling is hyperbolic (i.e., time and space scale in the same way), and for $s<1/2$ space scales faster than time.

The regularity of solutions to this problem was studied by Caffarelli and the second author in \cite{CF}.
Our goal here is to investigate the structure and regularity of the free boundary $\partial\{u=\varphi\}$.
Note that in the American option model the strategy changes discontinuously along the boundary of the exercise region $\{u=\varphi\}$, and thus it is important to understand the geometry and regularity of this set \cite{LS}.

Because the analysis of the regularity of the set $\partial\{u=\varphi\}$ is based on blow-up arguments, the way space and time rescale with respect to each other play a crucial role in the analysis.
As we shall explain in Section \ref{sect:main}, the most relevant regime for applications to finance is $s \in (\frac12,1)$, hence we shall focus on this case.
As explained in detail below, our main result establishes that the free boundary $\partial\{u=\varphi\}$ is $C^{1,\alpha}$ in $x$ and $t$ near regular points.

\subsection{Known results}

In the elliptic case ---which corresponds to the case $T=\infty$ in the optimal stopping model--- the regularity of solutions and free boundaries is quite well understood.
Indeed, by the results of Caffarelli-Salsa-Silvestre \cite{CSS}, solutions $u$ are $C^{1+s}(\R^n)$ and at any free boundary point $x_0\in \partial\{u=\varphi\}$ we have the following dichotomy:
\begin{itemize}
\item[(a)] either \quad $0<c\,r^{1+s}\leq \sup_{B_r(x_0)}(u-\varphi)\leq C\,r^{1+s}$ \vspace{2mm}
\item[(b)] or \quad $\sup_{B_r(x_0)}(u-\varphi)\leq C\,r^2$
\end{itemize}
Moreover, set of {regular points} (a) is an open subset of the free boundary, and it is locally a $C^{1,\alpha}$ graph.

After the results of \cite{CSS}, the set of singular points ---those at which the contact set has zero density--- was studied by Garofalo and Petrosyan in case $s=\frac12$ \cite{GP}.
Then, still when $s=\frac12$,  De Silva-Savin and Koch-Petrosyan-Shi proved that the regular set is $C^\infty$ \cite{DS,KPS}.
Under a superharmonicity assumption on the obstacle $\varphi$, the authors established in \cite{BFR} a complete characterization of free boundary points analogous to the one of the classical Laplacian, obtained in the seminal paper by Caffarelli \cite{C}.
Very recently, the results of \cite{CSS} have also been extended to more general nonlocal operators in \cite{CRS}.

Despite all these developments for the elliptic problem, much less is known in the parabolic  setting \eqref{obst-pb}.
The only result is due to Caffarelli and the second author: in \cite{CF}, they showed the optimal $C^{1+s}_x$ spatial regularity of solutions, as well as the $C^{\frac{1+s-\epsilon}{2s}}_t$ time regularity of solutions for all $\epsilon>0$.
However, nothing was known about the regularity of the free boundary in the parabolic setting.
The main reason for this lack of results is due to the fact that the approaches used in the stationary case completely fail in the evolutionary setting.
Indeed, the main tool to study the free boundary is based on classifications of blow-up profiles, and the papers \cite{CSS,GP,BFR} all use monotonicity-type formulas
that do not seem to exist in the parabolic setting.
Also, although the recent paper  \cite{CRS} circumvents the use of monotonicity formulas 
by combining Liouville and Harnack's type techniques, the methods there 
do not to apply in our context.
Hence, completely new ideas and techniques need to be introduced in the parabolic setting.

\subsection{Main result}
\label{sect:main}

Our main theorem extends the results of \cite{CSS} to the parabolic setting \eqref{obst-pb} when $s>\frac12$, and establishes the $C^{1,\alpha}$ regularity of the free boundary in $x$ and $t$ near regular points.
The result is new even in dimension $n=1$, and reads as follows (here and throughout the paper, we denote by $Q_r(x_0,t_0)=B_r(x_0)\times (t_0-r^{2s},t_0+r^{2s})$
the parabolic cylinder of size $r$ around $(x_0,t_0)$):

\begin{thm}\label{thm1}
Let $s\in(\frac12,1)$, let $\varphi\in C^4(\R^n)$ be an obstacle satisfying
\begin{equation}\label{obstacle}
\|D^k\varphi\|_{L^\infty(\R^n)}<\infty\quad \textrm{for}\quad 1\leq k\leq 4,
\end{equation}
and let $u$ be the solution of \eqref{obst-pb}.

Then, for each free boundary point $(x_0,t_0)\in \partial\{u=\varphi\}$, we have:
\begin{itemize}
\item[(i)] either
\[0<c\,r^{1+s}\leq \sup_{Q_r(x_0,t_0)}(u-\varphi)\leq C\,r^{1+s},\]
\item[(ii)] or
\[\hspace{40mm} 0\leq \sup_{Q_r(x_0,t_0)}(u-\varphi)\leq C_\epsilon\, r^{2-\epsilon}\qquad \textrm{for all}\ \,\epsilon>0.\]
\end{itemize}
Moreover, the set of points $(x_0,t_0)$ satisfying (i) is an open subset of the free boundary and it is locally a $C^{1,\alpha}$ graph in $x$ and $t$, for some small $\alpha>0$.

Furthermore, for any point $(x_0,t_0)$ satisfying (i) there is $r>0$ such that $u\in C^{1+s}_{x,t}(Q_r(x_0,t_0))$, and we have the expansion
\[u(x,t)-\varphi(x) = c_0\bigl((x-x_0)\cdot e+a(t-t_0)\bigr)_+^{1+s}+o\bigl(|x-x_0|^{1+s+\alpha} + |t-t_0|^{1+s+\alpha}\bigr),\]
for some $c_0>0$, $e\in \mathbb{S}^{n-1}$, and $a>0$.
\end{thm}

It is important to notice that the assumption $s>\frac12$ is necessary for the previous result to hold.
Indeed, by the examples constructed in \cite{CF}, the structure of the free boundary would be different when $s\leq \frac12$.
More precisely, it was shown in \cite[Remark 3.7]{CF} that if $s=\frac12$ then there are global solutions which are homogeneous of degree $1+\beta$ for any $\frac12\leq \beta<1$.
This means that when $s=\frac12$ there will be free boundary points satisfying neither (i) nor (ii), and there is no ``gap'' between the homogeneities $1+s$ and 2 as in Theorem~\ref{thm1}.

From the financial modeling point of view, the assumption $s>\frac12$ is natural.
For example, it was shown in \cite{Nature} that the scaling exponent of an economic index (Standard \& Poor's 500) is around $2s=1.4$ (remarkably constant) over the six-year period 1984-1989.
Furthermore, in American option models the obstacle (payoff) $\varphi$ has frequently linear growth at infinity \cite{LS,CT}, and in that case $s>\frac12$ is needed for problem \eqref{obst-pb} to be well posed.
Notice also that our assumption \eqref{obstacle} does allow the obstacle $\varphi$ to have linear growth at infinity.

\subsection{Related problems}

In the elliptic case, the obstacle problem for the fractional Laplacian is equivalent to a thin obstacle problem in $\R^{n+1}$, also known as the Signorini problem when $s=\frac12$.
A parabolic version of the Signorini problem has been recently studied in \cite{DGPT,ACM}.

We emphasize that, although the time-independent version of the problem studied in \cite{DGPT,ACM} is equivalent to the obstacle problem for the half-Laplacian, the parabolic problem is of completely different nature from the one considered in the present paper.
In particular, notice that for the parabolic Signorini problem in \cite{DGPT,ACM} one has Almgren-type and other monotonicity formulas (analogous to the elliptic ones used in \cite{CSS,GP}), while no such monotonicity formulas are known for our problem \eqref{obst-pb}.

\subsection{Structure of the paper}

The paper is organized as follows.
In Section \ref{sec-prelim} we prove the semiconvexity of solutions in $(x,t)$.
In Section \ref{sec-classif} we classify all global convex solutions to the obstacle problem with subquadratic growth at infinity.
In Section \ref{sec4} we show that, at any regular point, a blow-up of the solution $u$ converges in the $C^1$ norm to a global convex solution with subquadratic growth.
In Section \ref{sec5} we prove that that the free boundary is Lipschitz in $x$ and $t$ near regular points.
In Section \ref{sec6} we show that the regular set is open, and that it is $C^{1,\alpha}$ in $x$.
Finally, we prove in Section \ref{sec7} that the free boundary is $C^{1,\beta}$ in $x$ and $t$ near regular points, and in Section \ref{sec8} we establish Theorem \ref{thm1}.

\section{Preliminaries}
\label{sec-prelim}

In this Section we provide some preliminary results.
First, we establish the semiconvexity of solutions in $x$ and $t$.
The proof is similar to \cite[Theorem 2.1]{ACM} or \cite[Lemma 3.1]{CF}.

\begin{lem}[Semiconvexity in $(x,t)$]\label{semiconvex}
Let $\varphi$ be any obstacle satisfying \eqref{obstacle}, and $u$ be the solution to \eqref{obst-pb}.
Let $\xi=(\alpha e,\beta) \in \R^n\times \R$, with $e\in \mathbb{S}^{n-1}$ and $\alpha^2+\beta^2=1$.
Then, we have
\[u_{\xi\xi}:=\partial_{\xi\xi}u\geq -\hat C,\]
where constant $\hat C$ depends only on $\varphi$.
\end{lem}

\begin{proof}
We use a penalization method: it is well known that the solution $u$ can be constructed as the limit of $u^\varepsilon$ as $\varepsilon\rightarrow0$, where $u^\varepsilon$ are smooth solutions of
\[\begin{split}
\partial_t u^\varepsilon+(-\Delta)^su^\varepsilon&=\beta_\varepsilon(u_\varepsilon-\varphi)\quad\textrm{in}\ \, \R^n\times(0,T),\\
u^\varepsilon(\cdot,0)&=\varphi+\varepsilon\quad\quad\quad\,\textrm{at}\ \,t=0,
\end{split}\]
with $\beta_\varepsilon(z)=e^{-z/\varepsilon}$; see \cite[Lemma 3.1]{CF}.

Then, differentiating the equation twice and using that $\beta_\varepsilon''\geq0$, we get
\[\partial_t u_{\xi\xi}^\varepsilon+(-\Delta)^su_{\xi\xi}^\varepsilon \geq \beta_\varepsilon'(u^\varepsilon-\varphi) (u_{\xi\xi}^\varepsilon-\varphi_{\xi\xi})\quad\textrm{in}\ \, \R^n\times(0,T).\]
In particular, since $\beta_\varepsilon'\leq0$ we have
\[\partial_t (u_{\xi\xi}^\varepsilon+C_0)+(-\Delta)^s(u_{\xi\xi}^\varepsilon+C_0) \geq \beta_\varepsilon'(u^\varepsilon-\varphi) (u_{\xi\xi}^\varepsilon+C_0)\quad\textrm{in}\ \, \R^n\times(0,T),\]
where
\[\begin{split}
C_0&:= \|u^\varepsilon_{\xi\xi}(\cdot,0)\|_{L^\infty(\R^n)}\\
&\leq \alpha^2\|\partial_{ee}u^\varepsilon(\cdot,0)\|_{L^\infty(\R^n)}+ 2\alpha\beta\|\partial_e\partial_t u^\varepsilon(\cdot,0)\|_{L^\infty(\R^n)}+ \beta^2\|\partial_{tt} u^\varepsilon(\cdot,0)\|_{L^\infty(\R^n)}\\
&\leq \|D^2\varphi\|_{L^\infty(\R^n)}+ \|\nabla(-\Delta)^s\varphi\|_{L^\infty(\R^n)}+\|(-\Delta)^{2s}\varphi\|_{L^\infty(\R^n)}\\
&\leq C\Bigl(\|\nabla \varphi\|_{L^\infty(\R^n)}+\|D^2 \varphi\|_{L^\infty(\R^n)}+\|D^3 \varphi\|_{L^\infty(\R^n)}+
\|D^4 \varphi\|_{L^\infty(\R^n)}\Bigr)<\infty
\end{split}\]
thanks to \eqref{obstacle}.

Using again that $\beta_\varepsilon'\leq0$, it follows that $\beta_\varepsilon'(u^\varepsilon-\varphi)(u_{\xi\xi}^\varepsilon+C_0)\geq0$ whenever $u_{\xi\xi}^\varepsilon+C_0\leq0$.
Thanks to this fact, it follows that the function $w:=\min\{0,\,u_{\xi\xi}^\varepsilon+C_0\}$ satisfies
\[\partial_tw+(-\Delta)^sw\geq0\quad\textrm{in}\ \, \R^n\times(0,T).\]
Moreover,  by the definition of $C_0$, we have $w\equiv0$ at $t=0$.
Thus, by the minimum principle we get $w\geq0$, or equivalently $u_{\xi\xi}^\varepsilon+C_0\geq0$.
Letting $\varepsilon\to0$ we get the desired result.
\end{proof}

Throughout Sections \ref{sec-classif}, \ref{sec4}, \ref{sec5}, and \ref{sec6}, we will use the extension problem for the fractional Laplacian.
Namely, we will use that, for each fixed $t$, the function $u(x,t)$ can be extended to a function $u(x,y,t)$ satisfying
\[\left\{ \begin{array}{rcll}
u(x,0,t)&=&u(x,t) &\textrm{in}\ \mathbb{R}^n,\\
L_a u(x,y,t)&=& 0&\textrm{in}\ \mathbb{R}^{n+1}_{+},
\end{array}\right.\]
where $\R^{n+1}_+=\R^{n+1}\cap\{y>0\}$ and
\[L_a u:=\textrm{div}_{x,y}\bigl(y^{a}\nabla_{x,y} u\bigr),\qquad a=1-2s.\]
As shown in \cite{MO,CSext}, with this definition the fractional Laplacian can be computed as a (weighted) normal derivative of such extension $u(x,y,t)$, namely
\[\lim_{y\downarrow 0} y^{a}\partial_yu(x,y,t)=(-\Delta)^su(x,t)\quad \textrm{in}\ \mathbb{R}^n.\]
Therefore, our solution $u(x,y,t)$ to \eqref{obst-pb} satisfies
\[\begin{split}
L_au&=0\quad\textrm{in}\ \, \{y>0\}\times(0,T],\\
\min\bigl\{\partial_t u-\lim_{y\downarrow 0} y^{a}\partial_yu,\,u-\varphi\bigr\}&=0\quad\textrm{on}\ \, \{y=0\}\times(0,T],\\
u(\cdot,0,0)&=\varphi\quad \textrm{at}\ \,t=0.
\end{split}\]
Furthermore, given a free boundary point $(x_0,t_0)\in \partial\{u=\varphi\}$, we denote
\begin{equation}\label{v}
v(x,y,t):=u(x,y,t)-\varphi(x)+\frac{1}{4(1-s)}\,\Delta\varphi(x_0)\,y^2.
\end{equation}
With this definition it follows that $v=u-\varphi$ on $\{y=0\}$, and that
\begin{equation}\label{ec_v}
\begin{cases}
L_a v=y^{a}g(x) \quad &\mbox{in } \mathbb{R}^{n+1}_{+}\times[0,T]\setminus\{v(x,0,t)=0\},\\
v\ge 0,& \mbox{on } \{y=0\},\\
\lim_{y\downarrow 0} y^{a}\partial_yv=\partial_tv, & \mbox{on } \{v(x,0,t)>0\},\\
v(x,0,0)=0,
\end{cases}
\end{equation}
where $g(x):=\Delta\varphi(x)-\Delta\varphi (x_0)$. Also,
using the regularity of the obstacle (here we only need $\varphi\in C^{2,1}$), it follows that
\begin{equation}\label{g}
|g(x)|\leq C|x-x_0| \qquad \mbox{and} \qquad |\nabla g(x)|\leq c.
\end{equation}
Finally, throughout the paper, given $r \in (0,\infty]$, $\Q_r$ will denote the following (parabolic) cylinders in $\R^{n+1}_+$,
\[\Q_r(x_0,t_0):=\B_r(x_0)\times\bigl(t_0-r^{2s},t_0+r^{2s}\bigr),\qquad \textrm{and}\qquad \Q_r:=\Q_r(0,0),\]
while $Q_r$ will denote cylinders in $\R^n$,
\[Q_r(x_0,t_0):=B_r(x_0)\times\bigl(t_0-r^{2s},t_0+r^{2s}\bigr),\qquad \textrm{and}\qquad Q_r:=Q_r(0,0).\]
Here, $\B_r$ and $B_r$ denote balls in $\R^{n+1}_+$ and $\R^n$, respectively, i.e.,
\[\B_r(x_0):=\left\{(x,y)\in \R^{n+1}_+\,:\, |x-x_0|^2+y^2\leq r^2\right\},\qquad \B_r=\B_r(0),\]
\[ B_r(x_0):=\{x\in \R^n\,:\, |x-x_0|\leq r\},\qquad B_r=B_r(0).\]

\section{Classification of global convex solutions}
\label{sec-classif}

Because solutions to our problem are semiconvex in space-time (see Lemma \ref{semiconvex}),
the blow-up profiles that we shall consider will be convex in space-time.
Hence, it is natural to classify global convex solutions.

The main result of this section is the next theorem, which classifies all global convex solutions to the obstacle problem under a growth assumption on $u$.
Recall that $\Q_\infty=\{(x,y,t)\in \R^{n+1}_+\times(-\infty,\infty)\}$ and that $a=1-2s.$

\begin{thm}\label{thmclassif}
Let $s>\frac12$, and let $u\in C(\Q_\infty)$ satisfy
\begin{equation}\label{eqn-global}\begin{cases}
L_a u=0  \quad &\mbox{in } \Q_\infty\cap\{y>0\}\\
\min\left\{\partial_t u-\lim_{y\downarrow0}y^{a}\partial_yu,\,u\right\}=0 \quad &\mbox{on } \Q_\infty\cap\{y=0\}\\
D^2_{x,t}u \ge 0 \quad &\mbox{on } \Q_\infty\\
u\ge 0,\quad \partial_t u \ge 0 & \mbox{on } \Q_\infty\cap\{y=0\}.
\end{cases}\end{equation}
Assume in addition that $u(0,0,0)=0$, 
and that $u$ satisfies the growth control
\begin{equation}\label{growthctrolgradient}
\|u\|_{L^\infty(\Q_R)} \le R^{2-\epsilon} \quad \mbox{ for all }R\ge 1.
\end{equation}
Then, either $u \equiv 0$ or
\[ u(x,y,t) = K\, u_0(x\cdot e,y)\]
for some $e\in \mathbb{S}^{n-1}$ and $K>0$, where $u_0$ is the unique global solution to the elliptic problem for $n=1$ that is convex in the first variable
and satisfying $\|u_0\|_{L^\infty(\Q_1)}=1$.
Namely, $u_0$ is given by
\[u_0(z,y)=\frac{2^{-s}}{1-s}\bigl(\sqrt{z^2+y^2}+z\bigr)^s\bigl(z-s\sqrt{z^2+y^2}\bigr)\qquad \forall\,(z,y)\in \R^2_+,\]
and satisfies $u_0(z,0)=(z_+)^{1+s}$ on $\{y=0\}$.
\end{thm}

To prove it, we need some lemmas.
First, we show the following technical lemma.

\begin{lem}\label{lem-rescalings0}
Assume $w\in C(\Q_\infty)$ satisfies, for some $\mu>0$,
\[\|w\|_{L^\infty(\Q_R)} \leq R^\mu\quad \textrm{for all}\quad R\geq1.\]
Then, there is a sequence $R_k\to\infty$ for which the rescaled functions
\[w_k(x,y,t):=\frac{w(R_k x,R_k y, R_k^{2s}t)}{\|w\|_{L^\infty(\Q_{R_k})}}\]
satisfy
\[\|w_k\|_{L^\infty(\Q_R)}\leq 2R^\mu\quad\textrm{for all}\ R\geq1.\]
\end{lem}

\begin{proof}
Set
\[\theta(\rho):= \sup_{R\geq \rho} R^{-\mu}\|w\|_{L^\infty(\Q_R)}.\]
Note that, thanks to our assumption, $\theta$ is bounded by $1$ on $[1,\infty)$.

Since by construction $\theta$ is nonincreasing, for every $k \in \mathbb N$ there is  $R_k\ge k$ such that
\begin{equation}\label{nondeg-m}
 (R_k)^{-\mu} \|w\|_{L^\infty(\Q_{R_k})} \ge \frac 12 \theta(k)\ge \frac 12 \theta(R_k).
\end{equation}
With this choice we see that, for any $R\geq1$, we have
\[\|w_k\|_{L^\infty(\Q_R)}=\frac{\|w\|_{L^\infty(\Q_{R_kR})}}{\|w\|_{L^\infty(\Q_{R_k})}}
\leq \frac{\theta(R_kR)(R_kR)^\mu}{\frac12 \theta(R_k)(R_k)^\mu}\leq 2R^\mu,\]
where, in the last inequality, we used the monotonicity of $\theta$.
\end{proof}

We also need the following Liouville-type result.
\begin{lem}\label{lem-uniq-cont}
Let $u\in C(\overline{\R^{n+1}_+})$ be a function satisfying
\begin{equation}\label{eqn-global-w}
\begin{cases}
L_a u=0  \quad &\mbox{in } \R^{n+1}_+\\
D^2_x u\geq 0  \quad &\mbox{in } \R^{n+1}_+\\
|u(x,y)| \leq C(1+|x|+|y|)^{2-\epsilon}\quad &\mbox{in } \R^{n+1}_+\\
u \geq 0\quad &\mbox{on } \{y=0\}\\
\lim_{y\downarrow0}y^{a}\partial_y u \geq 0 \quad &\mbox{on } \{y=0\}\\
u(0,0)=0\\
\lim_{y\downarrow0}y^{a}\partial_y u(0,y)=0.
\end{cases}
\end{equation}
Then $u\equiv0$.
\end{lem}

\begin{proof}
We begin by noting that combining the equation $L_a u=0$ with the convexity of $u$ in $x$, it follows that
\begin{equation}
\label{eq:y concave}
\partial_y(y^a \partial_y u)=-y^a\Delta_x u\leq 0\qquad \mbox{in } \R^{n+1}_+.
\end{equation}
Thanks to this fact, fixed $R>0$, for any $x \in \R^n$ and $y \in [0,R]$ we have
\begin{multline*}
u(x,2R)-u(x,y)=\int_y^{2R}z^a \partial_y u(x,z)\frac{dz}{z^a}\leq y^a  \partial_y u(x,y)\int_y^{2R}\frac{dz}{z^a}\\ \leq y^a  \partial_y u(x,y)\int_0^{2R}\frac{dz}{z^a}
=(1-a) 2^{1-a}\,R^{1-a}\,y^a  \partial_y u(x,y).
\end{multline*}
Hence, if we set $v(x,y):=y^a\partial_y u(x,y)$, combining the above estimate with the third and fifth property in \eqref{eqn-global-w} we deduce that
$$
v\geq 0\quad \mbox{on } \{y=0\},\qquad v \geq -C_a\,R^{1+a-\epsilon} \qquad \mbox{on } \partial\bigl(B_R\times[0,R]\bigr)\cap\{y>0\},
$$
where $C_a>0$ is independent of $R$.
Also, since $L_a u=0$, it follows by a direct computation that $L_{-a}v=0$.

Consider now the barrier
$$
b_R(x,y):=-\frac{n+1}{1-a}\,y^{1+a} - \frac{|x|^2 - \frac{n}{1-a}\,y^2}{R^{1-a}}.
$$
We note that $L_{-a}b_R=0$ and
$$
b_R= 0\quad \mbox{on } \{y=0\},\qquad b_R \leq - R^{1+a} \qquad \mbox{on } \partial\bigl(B_R\times[0,R]\bigr)\cap\{y>0\},
$$
Hence, given $\delta>0$, it follows by the maximum principle that, for all $R\geq R_\delta$ sufficiently large,
$$
v \geq \delta\, b_R\qquad \text{in }B_R\times[0,R].
$$
Letting $R\to \infty$ this implies that
$$
v \geq -\delta\,\frac{n+1}{1-a}\,y^{1+a}\qquad \mbox{in } \R^{n+1}_+,
$$
so, by letting $\delta \to 0$, we deduce that $v \geq 0$ in $\R^{n+1}_+$.

On the other hand,  it follows by \eqref{eq:y concave} and the last property in \eqref{eqn-global-w}
that $v(0,y)=y^{a}\partial_y u(0,y) \leq 0$ for all $y \geq 0$,
thus $v(0,y)= 0$ for all $y \geq 0$.

This proves that $v$ is a non-negative solution of $L_{-a}v=0$ in $\R^{n+1}_+$ that vanishes at some interior point, hence it is identically zero by the strong maximum principle.

Since $v \equiv 0$ we deduce that $\partial_y u\equiv 0$. Hence, by the forth and sixth property in \eqref{eqn-global-w},
it follows that $u \geq 0$ in $\R^{n+1}_+$ and  $u(0,y) = 0$ for all $y \geq 0$.
Since $L_au=0$ in $\R^{n+1}_+$, applying again the strong maximum principle we obtain that $u \equiv 0$, as desired.
\end{proof}

We can now prove the main result of this section.

\begin{proof}[Proof of Theorem \ref{thmclassif}]
If $u\equiv0$ then there is nothing to prove.
Hence, we assume that $u$ is not identically zero.

The key step in the proof is the following:

\vspace{3mm}

\noindent \textbf{Claim.} The contact set $\{u=0\}\cap \{y=0\}$ contains a line of the form $\{(x,t)\,:\,x=x_0\text{ for some }x_0\in \R^n\}.$

\vspace{3mm}

Let us prove it by contradiction.
Assume the Claim is not true, and let
\[\Lambda:=\{u=0\}\cap \{y=0\}.\]
Then, since $u$ is convex in space-time, also the set $\Lambda$  is convex in the $(x,t)$-space.
Hence,  there exist $p\in \R^n$ and some $\kappa\in\R$ such that  $\Lambda\subset \{t\leq x\cdot p+\kappa\}$.

We now perform blow-down of our solution using a parabolic scaling (recall $s>\frac12$), and we show that we get a solution to the same problem but with contact set contained in $\{t\leq0\}$.
Indeed, let us consider the rescaled functions
\[U_k(x,y,t):=\frac{u(R_kx,R_ky,R_k^{2s}t)}{\|u\|_{\Q_{R_k}}},\]
with $R_k\to\infty$ given by Lemma \ref{lem-rescalings0}.
Then, the functions $U_k\geq0$ are convex in $x$ and~$t$, and satisfy (recall that $a=1-2s$)
\begin{equation}\label{eqn-global-k}\begin{cases}
L_a U_k=0  \quad &\mbox{in } \Q_\infty\cap\{y>0\}\\
\partial_tU_k=\lim_{y\downarrow0}y^{a}\partial_yU_k \quad &\mbox{on } (\Q_\infty\cap\{y=0\})\setminus\Lambda_k\\
U_k=0 \quad &\mbox{on } \Lambda_k\\
\partial_t U_k \ge 0 & \mbox{on } \Q_\infty\cap\{y=0\},
\end{cases}\end{equation}
$U_k(0,0,0)=0$, 
$\|U_k\|_{L^\infty(\Q_1)}=1$, and
\[\|U_k\|_{L^\infty(\Q_R)} \le 2R^{2-\epsilon} \quad \mbox{ for all }R\ge 1.\]
Moreover, we have
\begin{equation}
\label{eq:lambda k}\Lambda_k \subset \{R_k^{2s}t\leq R_k x\cdot p+\kappa\}=\{t\leq R_k^{a} x\cdot p+R_k^{-2s}\kappa\}.
\end{equation}
By the $C^{1+\alpha}$ regularity estimates of \cite{CF}, a subsequence of the functions $U_k$ converge in $C^1_{\rm loc}$ to a nontrivial solution $U_\infty$ to the same equation satisfying $U_\infty(0,0,0)=0$ 
and $\|U_\infty\|_{L^\infty(\Q_1)}=1$.
Also, because $U_k$ are obtained as blow-downs of the convex function $u$,
it follows from \eqref{eq:lambda k} that  $\Lambda_\infty\subset \{t\leq 0\}$ (recall that $a=1-2s<0$.
\vspace{3mm}

To see that this is not possible, we define $w(x,y):=U_\infty(x,y,0)$ and we claim that $w$ satisfies
all the assumptions in \eqref{eqn-global-w}.
Indeed, all the properties except the fifth and the last one follow easily from the construction of $U_\infty.$
To check the other two properties we notice that, since $U_\infty$ satisfies \eqref{eqn-global-k} and $\Lambda_\infty\subset \{t\leq 0\}$,
$$
\lim_{y\downarrow0}y^{a}\partial_y U_\infty(x,y,t)= \partial_t U_\infty(x,y,t) \geq 0\qquad \forall\,t>0.
$$
Also, since $U_\infty \geq 0$ and $U_\infty(0,0,0)=0$, we deduce that $\partial_tU_\infty(0,0,0)=0$.
 Hence,
it follows by the $C^{1+\alpha}$ regularity estimates of \cite{CF} that
$$
\lim_{y\downarrow0}y^{a}\partial_y w(x,y)=\lim_{t \downarrow 0} \partial_t U_\infty(x,y,t)\geq 0,
$$
and
$$
\lim_{y\downarrow0}y^{a}\partial_y w(0,y)=\partial_t U_\infty(0,0,0)=0,
$$
as desired.

This allows us to apply Lemma \ref{lem-uniq-cont} to $w$ and deduce that
$w\equiv 0$. This proves that $U_\infty=0$ at $t=0$. Hence, since $U_\infty$ solves the ``extension version'' of the fractional heat equation, by uniqueness of solutions
we deduce that $U_\infty \equiv 0$ for all $t \geq 0.$ On the other hand, since $\partial_t U_\infty \geq 0$ and $U_\infty \geq 0$,
we get $U_\infty \equiv 0$ for all $t \leq 0.$
This proves that $U_\infty\equiv 0$ in $\Q_\infty$, a contradiction to the fact that $\|U_\infty\|_{L^\infty(\Q_1)}=1$.

Thus, the Claim is proved.

\vspace{3mm}

Using the Claim, we notice that $u$ is a convex function in $x$ and $t$ that vanishes on a line of the form $\{x=x_0\}$.
This implies that $u$ is independent of $t$, thus $u(x,y,t)=u(x,y)$.
By the (elliptic) classification result in \cite[Section 5]{CSS}, we get the desired result.
\end{proof}

\section{Regular points and blow-ups}
\label{sec4}

The aim of this Section is to prove that, whenever (ii) in Theorem \ref{thm1} does \emph{not} hold, then a blow-up of $u(x,t)$ at $(x_0,t_0)$ converges in the $C^1$ norm to the 1D solution $(x\cdot e)_+^{1+s}$ for some $e\in \mathbb{S}^{n-1}$.

Recall that we denote
\[\Q_r(x_0,t_0)=\B_r(x_0)\times\bigl(t_0-r^{2s},t_0+r^{2s}\bigr),\qquad \textrm{and}\qquad \Q_r=\Q_r(0,0).\]
According to Theorem \ref{thm1}, we next define regular free boundary points.

\begin{defi}\label{def-regular}
We say that a free boundary point $(x_0,t_0) \in \partial \{u=\varphi\}$ is \emph{regular} if
\begin{equation}\label{regular-point0}
\limsup_{r\downarrow0}\frac{\|u-\varphi\|_{L^\infty(\Q_r(x_0,t_0))}}{r^{2-\epsilon}} = \infty
\end{equation}
for some $\epsilon>0$.
Notice that if a free boundary point $(x_0,t_0)$ is not regular, then $\|u-\varphi\|_{L^\infty(\Q_r(x_0,t_0))}=O(r^{2-\epsilon})$ for all $\epsilon>0$, so (ii) in Theorem~\ref{thm1} holds.
\end{defi}

The definition of \emph{regular} free boundary point is qualitative.
We will also need the following quantitative version.

\begin{defi}\label{def-regular-2}
Let $\nu:(0,\infty) \rightarrow (0,\infty)$ be a nonincreasing function with
\[\lim_{\rho\downarrow 0} \nu(\rho)= \infty.\]
Given $\epsilon>0$, we say that a free boundary point $(x_0,t_0) \in \partial \{u=\varphi\}$ is \emph{regular}  with exponent~$\epsilon>0$ and modulus $\nu$ if
\begin{equation}\label{regular-point}
\sup_{r\ge \rho} \frac{\|u-\varphi\|_{L^\infty(\Q_{r}(x_0,t_0))}}{r^{2-\epsilon}} \ge \nu(\rho).
\end{equation}
\end{defi}

The main result of this section is the following.
It states that at any regular free boundary point $(x_0,t_0)$ there is a blow-up sequence that converges to $(e\cdot x)_+^{1+s}$ for some $e\in \mathbb{S}^{n-1}$.

\begin{prop}\label{proprescalings}
Let $\varphi\in C^4(\R^n)$ be any obstacle satisfying \eqref{obstacle}, and $u$ be the solution to \eqref{obst-pb}, with $s\in(\frac12,1)$.

Assume that $(x_0,t_0)$ is a regular free boundary point with exponent~$\epsilon>0$ and modulus~$\nu$.
Then, given $\delta>0$ and $r_0>0$, there is
\[ r = r(\delta,\epsilon,\nu,r_0,n,s,\varphi)\in (0,r_0)\]
such that $\|u-\varphi\|_{L^\infty(\Q_r(x_0,t_0))}\geq \frac12r^{2-\epsilon}$ and the rescaled function
\[ u_r(x,y,t) : = \frac{(u-\varphi)(x_0+r x,ry,t_0+r^{2s}t)}{\|u-\varphi\|_{L^\infty(\Q_r(x_0,t_0))}}\]
satisfies
\begin{equation}
\bigl| u_r(x,y,t) - u_0(x\cdot e,y) \bigr|  +\bigl |\nabla u_r -  \nabla u_0\bigr|+|\partial_t u_r| \le \delta \quad \mbox{in } \Q_1
\end{equation}
for some $e\in \mathbb{S}^{n-1}$.
Here, $u_0=u_0(x\cdot e,y)$ is the unique global solution given by the classification Theorem \ref{thmclassif}. 
\end{prop}

For this, we will need the following result, whose proof is essentially the same of the one of Lemma \ref{lem-rescalings0}.

\begin{lem}\label{lem-rescalings}
Assume $w\in L^\infty(\Q_1)$ satisfies $\|w\|_{L^\infty(\Q_1)}=1$ and, for some $\mu>0$,
\[\qquad \qquad \sup_{\rho\leq r\leq 1}\frac{\|w\|_{L^\infty(\Q_r)}}{r^\mu}\geq \nu(\rho)\to\infty\qquad \textrm{as}\quad \rho\to0.\]
Then, there is a sequence $r_k\downarrow0$ for which $\|w\|_{L^\infty(\Q_{r_k})}\geq \frac12r_k^\mu$, and for which the rescaled functions
\[w_k(x)=\frac{w(r_k x,r_k y, r_k^{2s}t)}{\|w\|_{L^\infty(\Q_{r_k})}}\]
satisfy
\[\|w_k\|_{L^\infty(\Q_R)}\leq 2R^\mu\quad\textrm{for all}\ 1\leq R\leq \frac{1}{r_k}.\]
Moreover, $1/k\leq r_k\leq (\nu(1/k))^{-1/\mu}$.
\end{lem}

\begin{proof}
Defining
\[\theta(\rho):= \sup_{\rho\leq r\leq 1} r^{-\mu}\|w\|_{L^\infty(\Q_r)},\]
we note that $\theta$ is nonincreasing and that, by our assumption,
\[ \theta(\rho)\geq\nu(\rho)\rightarrow \infty\qquad \textrm{as}\quad \rho\downarrow0.\]
Hence, for every $k \in \mathbb N$ it suffices to choose $r_k\ge \frac1k$ such that
$$
 (r_k)^{-\mu} \|w\|_{L^\infty(\Q_{r_k})} \ge \frac 12 \theta(1/k)\ge \frac 12 \theta(r_k),
$$
and one concludes as in the proof of Lemma \ref{lem-rescalings0}.
%
\end{proof}

To prove Proposition \ref{proprescalings} we will also need the following result, that follows by compactness from Theorem \ref{thmclassif}.

\begin{lem}\label{lemconvergence}
Given $\delta>0$, there is
\[\eta=\eta(\delta,\epsilon, n,s)>0\]
such that the following statement holds:

Let $v:\Q_{1/\eta}\to \R$ satisfy $v(0,0)=0$, $\nabla v(0,0)=0$,
\begin{equation}\label{eq-comp}\begin{cases}
|L_a v|\leq \eta  \quad &\mbox{in } \Q_{1/\eta}\cap \{y>0\}\\
\min\bigl\{\partial_t v-\lim_{y\downarrow0}y^{1-2s}\partial_yv,\,v\bigr\}=0 \quad &\mbox{on } \Q_{1/\eta}\cap\{y=0\}\\
v_{\xi\xi} \ge -\eta \quad &\mbox{on } \{y=0\}\\
v\ge 0,\quad \partial_t v \ge 0 & \mbox{on } \{y=0\}
\end{cases}\end{equation}
with
\begin{equation}\label{gr-comp}
\|v\|_{L^\infty(\Q_{R})}  \le R^{2-\epsilon} \quad \mbox{for all }1\leq R\leq 1/\eta,
\end{equation}
and
\begin{equation}\label{nondeg-comp}
\|v\|_{L^\infty(\Q_{1})}=1.
\end{equation}

Then,
\[ \bigl| v(x,y,t) - u_0(x\cdot e,y)\bigr|+\bigl|\nabla v -  \nabla u_0\bigr|+|\partial_tv| \le \delta \quad \mbox{in}\quad \Q_1  \]
for some $e\in \mathbb{S}^{n-1}$.
\end{lem}

\begin{proof}
The proof is by compactness and contradiction.
Assume that for some $\delta>0$ we have sequences $\eta_k\downarrow 0$ and $v_k$ satisfying $v_k(0,0)=0$, $\nabla v_k(0,0)=0$, \eqref{eq-comp}, \eqref{gr-comp}, \eqref{nondeg-comp}, but
\begin{equation}\label{contr}
\bigl| v_k(x,y,t) - u_0(x\cdot e,y)\bigr|+\bigl|\nabla v_k - \nabla u_0\bigr|+|\partial_t v_k| \ge \delta \quad\textrm{in}\quad \Q_1 \quad \mbox{for all } e\in \mathbb{S}^{n-1}.
\end{equation}
By the regularity estimates in \cite{CF}, we have
\[ \|v_k\|_{C^{1,\alpha}_{x,t}(\Q_{R})}  \le C(R) \quad \mbox{for all }R\ge 1,\]
with $C(R)$ depending on $R$ but independent of $k$.
Thus, up to taking a subsequence, the functions $v_k$ converge in $C^1_{\rm loc}$ to a function $v_\infty$ that solves \eqref{eqn-global}, \eqref{growthctrolgradient}, \eqref{nondeg-comp}, $v_\infty(0,0)=0$ and $\nabla v_\infty(0,0)=0$.

Since $\|v_\infty\|_{L^\infty(\Q_{1})} = 1$, it follows by the classification result in Theorem \ref{thmclassif} that
\[ v_\infty(x,y,t) \equiv u_0(x\cdot e,y),\quad \mbox{for some } e\in \mathbb{S}^{n-1}.\]
This proves that $v_k\rightarrow u_0(x\cdot e,y)$ in the $C^1_{\rm loc}$ norm, which contradicts \eqref{contr} for $k$ large enough.
\end{proof}

We can now prove Proposition \ref{proprescalings}.

\begin{proof}[Proof of Proposition \ref{proprescalings}]
We may assume that $\|u-\varphi\|_{L^\infty(\Q_1(x_0,t_0))}=1$, and let $v$ be given by \eqref{v}.

Let $\eta=\eta(\delta,\epsilon, n,s)>0$ be the constant given by Lemma \ref{lemconvergence}, let $r_k$ be the sequence given by Lemma \ref{lem-rescalings} with $\mu=2-\epsilon$, and set
\[v_k(x,y,t) := \frac{v(x_0+r_k x,r_k y,t_0+r_k^{2s}t)}{\|v\|_{L^\infty(\Q_{r_k}(x_0,t_0))}}.\]
Then, recalling \eqref{ec_v} and \eqref{g}, the functions $v_k$ satisfy
\begin{equation}\label{A1}\begin{cases}
L_a v_k=g_k  \quad &\mbox{in } \{y>0\}\\
\min\bigl\{\partial_t v_k-\lim_{y\downarrow0}y^{1-2s}\partial_yv_k,\,v_k\bigr\}=0 \quad &\mbox{on } \{y=0\}\\
v_k\ge 0,\quad \partial_tv_k \ge 0 & \mbox{on } \{y=0\}
\end{cases}\end{equation}
with
\[|g_k(x)|=\frac{(r_k)^2|\Delta \varphi(x_0+r_k x)-\Delta\varphi(x_0)|}{\|v\|_{L^\infty(\Q_{r_k}(x_0,t_0))}}\leq \frac{C(r_k)^2}{(r_k)^{2-\epsilon}}\leq C(r_k)^\epsilon,\]
with $C$ depending only on $\varphi$.

Moreover, by Lemma \ref{semiconvex}, for any $e\in \mathbb{S}^{n-1}$
\[\partial_{ee}v_k(x,y,t)=\frac{(r_k)^2\partial_{ee}v(r_kx,r_ky,r_k^{2s}t)}{\|v\|_{L^\infty(\Q_{r_k}(x_0,t_0))}} \geq -\hat C(r_k)^\epsilon,\]
and
\[\partial_{tt}v_k(x,y,t)=\frac{(r_k)^{4s}\partial_{tt}v(r_kx,r_ky,r_k^{2s}t)}{\|v\|_{L^\infty(\Q_{r_k}(x_0,t_0))}} \geq -\hat C(r_k)^{4s+\epsilon-2}\geq -\hat C(r_k)^\epsilon.\]
on $\{y=0\}$.
Similarly, for any $\xi=\alpha e+\beta t$, with $|\alpha|+|\beta|=1$, we get
\[\partial_{\xi\xi}v_k\geq -\hat C(r_k)^\epsilon\quad \textrm{on}\quad \{y=0\}.\]
Furthermore, we have
\begin{equation}\label{A2}
\|v_k\|_{L^\infty(\Q_{R})}  \le R^{2-\epsilon} \quad \mbox{for all }1\leq R\leq 1/r_k,
\end{equation}
and
\begin{equation}\label{A3}
\|v_k\|_{L^\infty(\Q_{1})}=1,\qquad v_k(0,0)=0,\quad \nabla v_k(0,0)=0.
\end{equation}
Therefore, taking $k$ large enough, by Lemma \ref{lemconvergence} we obtain
\[\bigl| v_k(x,y,t) - u_0(x\cdot e,y) \bigr|  +\bigl |\nabla v_k -  \nabla u_0\bigr|+|\partial_tv_k| \le \delta \quad \mbox{in } \Q_1\]
for some $e\in \mathbb{S}^{n-1}$.
Notice that, thanks to Lemma \ref{lem-rescalings}, it suffices to take $k$ large enough so that
\[(r_k)^\epsilon \leq (\nu(1/k))^{-1/(2-\epsilon)}\leq \eta,\]
where $\eta$ is given by Lemma \ref{lemconvergence}.
In particular, the scaling parameter $r$ can be taken depending only on $\delta$, $n$, $s$, $r_0$, $\varphi$, $\epsilon$, and the modulus $\nu$.
\end{proof}

\section{Lipschitz regularity of the free boundary in $x$ and $t$}
\label{sec5}

The aim of this Section is to prove the Lipschitz regularity in $x$ of the free boundary in a neighborhood (in $x$ and $t$) of any regular free boundary point $(x_0,t_0)$.
In fact, the result gives also the $C^1_x$ regularity of the free boundary at the point $(x_0,t_0)$.

Let be $(x_0,t_0)$ a regular point of the free boundary.  Along this section,
$v$ will denote the function defined in \eqref{v}. Recall that $v$ satisfies \eqref{ec_v}.

The main result of this section is the following.

\begin{prop}\label{free-bdry-Lipschitz}
Assume that $(x_0,t_0)$ is a regular free boundary point with exponent $\epsilon>0$ and modulus~$\nu$, and let $v$ be the function defined in \eqref{v}. Then, there is $e\in \mathbb{S}^{n-1}$ such that for any $\ell\in(0,1)$ there exists $r>0$ such that
\begin{equation}\label{5-1}
\partial_{e'}v\geq0\quad\textrm{in}\quad \Q_r(x_0,t_0),\quad\textrm{for all}\quad e' \in \mathbb S^{n-1} \text{ with } e'\cdot e\geq\frac{\ell}{\sqrt {1+\ell^2}}.
\end{equation}
Moreover, we have
\begin{equation}\label{5-2}
\partial_{e'}v\geq \frac18|\nabla_x v|\quad\textrm{in}\quad \Q_r(x_0,t_0),\quad \mbox{for all}\quad e' \in \mathbb S^{n-1} \text{ with }e'\cdot e\geq\frac12.
\end{equation}
Furthermore, given $\eta>0$ and $\kappa>0$, the radius $r>0$ can be taken such that the rescaled function
\begin{equation}\label{xr}
v_r(x,y,t)=\frac{v(x_0+rx,ry,t_0+r^{2s}t)}{\|v\|_{L^\infty(\Q_r(x_0,t_0))}}
\end{equation}
satisfies
\begin{equation}\label{5-3}
0<\gamma\,\partial_e v_r\leq  \partial_t v_r\leq \kappa\,\partial_e v_r\quad \textrm{in}\quad \Q_1,
\end{equation}
and
\begin{equation}\label{5-4}
\partial_{e}v_r\geq c_1>0\quad \textrm{in}\quad \Q_1\cap\{(x-x_0)\cdot e\geq \eta\},
\end{equation}
\begin{equation}\label{5-5}
\partial_e v_r\geq c_2\,y^{2s}\quad \textrm{in}\quad \Q_1.
\end{equation}
Here, the constant $r>0$ depends only on $\ell$, $\kappa$, $\eta$, $\nu$, $n$, and $s$;
the constant $c_1>0$ depends only on $\ell$, $\eta$, $\nu$, $n$, and $s$;
the constant $c_2>0$ depends only on $\ell$, $\nu$, $n$, and $s$; and
the constant $\gamma>0$ depends on $u$ and the free boundary point $(x_0,t_0)$.
\end{prop}

As a direct consequence of Proposition \ref{free-bdry-Lipschitz}, we find the following.

\begin{cor}\label{cor-Lip}
Let $\varphi\in C^4(\R^n)$ be any obstacle satisfying \eqref{obstacle}, and $u$ be the solution to \eqref{obst-pb}, with $s\in(\frac12,1)$.
Assume that $(x_0,t_0)$ is a regular free boundary point with exponent~$\epsilon>0$ and modulus~$\nu$.

Then, there is $r>0$ such that the free boundary is Lipschitz in~$x$ and $t$ in $Q_r(x_0,t_0)$.
More precisely, after a rotation in the $x$-variables, we have
\[\partial \{u(x,t)=\varphi(x)\}\cap Q_r(x_0,t_0)\equiv\{x_n= G(x',t)\}\cap Q_r(x_0,t_0),\]
where $x=(x',x_n) \in \R^{n-1}\times \R$, and $G:\R^{n-1}\times \R\to \R$ is Lipschitz.

Furthermore, the free boundary is $C^1$ in $x$ at the point $(x_0,t_0)$, in the sense that for any $\ell>0$ there exists $r=r(\ell, \epsilon, \nu, n, s)>0$ such that
\[[G]_{{\rm Lip}_x(Q_r(x_0,t_0))}\leq \ell.\]
\end{cor}

\begin{proof}
The result follows from Proposition \ref{free-bdry-Lipschitz}.
Indeed, \eqref{5-1} implies that the level sets of the function $u-\varphi$ are $\ell$-Lipschitz in $x$, while \eqref{5-3} implies that
the level sets of the function $u-\varphi$ are uniformly Lipschitz in $t$.
\end{proof}

To prove Proposition \ref{free-bdry-Lipschitz} we will need the following parabolic version of \cite[Lemma~7.2]{CSS}.

\begin{lem}\label{lem-max-princ}
Let be $\Gamma\subseteq Q_1\subseteq \mathbb{R}^{n}\times [0,T]$,
set $c_{n,a}:=\min\left\{\frac18 \sqrt{\frac{s(1+a)}{n}},\left(\frac{\sqrt{s}}8\right)^{1/s}\right\}$,
and let $h:\Q_1\to \R$ be a continuous function satisfying the following properties
for some positive constants $\gamma$, $c_0$, and $\theta$:
\begin{itemize}
\item[(H1)]{$\displaystyle|L_a h|\leq \gamma\, y^{a}$ in $\Q_{1}\cap\{y>0\}$.}
\item[(H2)]{$\lim_{y\downarrow 0}y^{a}\partial_yh=\partial_th$ in $Q_{1}\setminus \Gamma.$}
\item[(H3)]{$h\geq0$ on $\Gamma$.}
\item[(H4)]{$h>-\theta$ on $\Q_1\cap \left\{0<y< c_{n,a}\right\}$.}
\item[(H5)]{$h\geq c_0$ on $\Q_1\cap \left\{y\geq  c_{n,a}\right\}$.}
\end{itemize}
If $\gamma \leq c_0$ and $\theta \leq \frac{s\,c_0}{64}$, then
\begin{equation}\label{hopf}
\mbox{$h\geq c_0 y^{2s}$ in $\Q_{1/2}$.}
\end{equation}
\end{lem}

\begin{proof}
We prove \eqref{hopf} by contradiction.
Hence, we suppose there exists $(x_0,y_0,t_0)\in \Q_{1/2}$ such that $h(x_0,y_0,t_0)<c_0 y_0^{2s}$. Notice that, thanks to (H5), $y_0 < c_{n,a}$.
Hence, we define
$$\Q:=\left\{(x,y,t):\, |x-x_0|<\frac{1}{4},\quad t_0-\frac{1}{4}<t<t_0,\quad 0<y<c_{n,a}\right\},$$
we consider the $a$-harmonic polynomial $P$ given by
$$P(x,y,t):=|x-x_0|^{2}+2s(t_0-t)-\frac{n}{a+1}y^2-y^{2s},$$
and we set
$$w(x,y,t):=h(x,y,t)+\tau P(x,y,t)-\frac{\gamma}{2(a+1)}y^2,$$
where $\gamma>0$ is as in (H1).
Then, thanks to (H1)-(H3), since $a=1-2s$ and $\partial_tP=
\lim_{y\downarrow 0} y^{a}\partial_y P$, we have that
\begin{equation}\label{eq-w h}\begin{cases}
L_a w=L_a h-\gamma \,y^{a}\leq 0 &\mbox{ in $\Q$}\\
w\geq \tau P>0& \mbox{ on $\Gamma$}\\
\lim_{y\downarrow 0} y^{a}\partial_y w =\partial_tw& \mbox { in $(\Q\cap\{y=0\})\setminus \Gamma$}\\
w(x_0,y_0, t_0)\leq h(x_0,y_0,t_0)-\tau \,y_0^{2s}<0.
\end{cases}\end{equation}
Since $(x_0,y_0,t_0) \in \Q$, it follows by the maximum principle  that $w$ must have a negative minimum at some point $(x_1,y_1, t_1)$ that belongs to the parabolic boundary $\partial_{P}\Q$ of $\Q$. 
Moreover, by the second and third equations in \eqref{eq-w h}, we deduce that $w(x,0,t)$ can attain its minimum only on the parabolic boundary of $\Q\cap\{y=0\}$.
Therefore, we deduce that
$(x_1,y_1,t_1)\in \overline{\partial_{P} \Q\cap\{y>0\}}$.

We now study now the sign of the function $w$ in each part of $\overline{\partial_{P} \Q\cap\{y>0\}}$ to get a contradiction. Notice that, with our choice of $c_{n,a}$,
\begin{equation}
\label{eq:y cna}
\frac{n}{a+1}y^2+y^{2s} \leq \frac{s}{64}+\frac{s}{64}= \frac{s}{32}\qquad \forall\,y \in [0,c_{n,a}]
\end{equation}

- If $y= c_{n,a}$,
it follows by (H5) and \eqref{eq:y cna} that
\begin{equation}\label{main}
w\geq c_0-\frac{s\,c_0}{32}-\frac{s\,\gamma}{128n}>0,
\end{equation}
provided $\gamma \leq c_0$.

- If $|x-x_0|=1/4$ and $y \in \left[0,c_{n,a}\right]$, then it follows by (H4) and
\eqref{eq:y cna} that
\begin{equation}\label{boundary_x}
w\geq -\theta+ c_0\left(\frac{1}{16}-\frac{s}{32}\right)-\frac{s\,\gamma}{128n}>0
\end{equation}
provided $\gamma \leq c_0$ and $\theta \leq \frac{c_0}{64}.$

- If $t=t_0-1/4,$ using again (H5) we obtain that
$$
w\geq -\theta+ c_0\left(\frac{s}{2}-\frac{s}{32}\right)-\frac{s\,\gamma}{128n}>0
$$
provided $\gamma \leq c_0$ and $\theta \leq \frac{s\,c_0}{4}$

Hence, if $\gamma \leq c_0$ and $\theta \leq \frac{s\,c_0}{64}$,
this shows the desired contradiction
provided $\gamma \leq c_0$ and $\theta \leq \frac{s\,c_0}{64}$,
concluding the proof.

%
%
\end{proof}

We now prove Proposition \ref{free-bdry-Lipschitz}.

\begin{proof}[Proof of Proposition \ref{free-bdry-Lipschitz}]
Given $\eta>0$ and $\kappa>0$, fix $\delta \in (0,\eta^s)$.

Consider the rescaled function $v_r$ defined in \eqref{xr} where
$r>0$ is given by Proposition \ref{proprescalings} and $v$ is defined in \eqref{v}. Thus, it follows that for some $e\in \mathbb{S}^{n-1}$
\[|\nabla v_r(x,y,t)-\nabla u_0(x,y)|+|\partial_t{v_r}(x,y,t)| \leq \delta\quad \textrm{in}\ \Q_1.\] Let us fix consider $\ell>0$ small and $e'\in \mathbb{S}^{n-1}$ such that
$$e'\cdot e \ge \frac{\ell}{\sqrt {1+\ell^2}} \ge \frac \ell2.$$
Then,
\begin{equation} \label{largeatmanypoints}
\partial_{e'}v_r\geq \partial_{e'}u_0 - \delta,\quad \kappa\,\partial_{e'}v_r -\partial_{t}v_{r}\geq \kappa\,\partial_{e'}u_0 - \delta \qquad \mbox{in }\Q_1.
\end{equation}
In particular we get that
$$\partial_{e}v_r\geq ((x-x_0)\cdot e)_+^s - \delta\geq \eta^s-\delta \quad \textrm{in}\quad \Q_1\cap\{(x-x_0)\cdot e\geq \eta\},$$
thus \eqref{5-4} is satisfied with $c_1:=\eta^s-\delta>0$.
\vspace{3mm}

Denoting by $C_r:=\|v\|_{L^\infty(\Q_{r}(x_0,t_0))}$, it follows by Proposition \ref{proprescalings} that
\begin{equation}\label{lower}
C_r\geq \frac12 r^{2-\epsilon}
\end{equation}
where $\epsilon>0$.
Moreover
$$
L_a v_r=\frac{r^{2-a}}{C_r} (L_a v)(x_0+r x,r y,t_0+r^{2s}t).
$$
Also,  recalling \eqref{ec_v} and that $1-a=2s$, we see that on the set $\{v_r(x,0,t)>0\}$ it holds
\begin{eqnarray*}
 \lim_{y\downarrow 0}y^{a}\partial_{y}v_r &=& \frac{r}{C_r} \lim_{y\downarrow 0}\left(y^{a}(\partial_{y} v)(x_0+r x,ry,t_0+r^{2s}t)\right)\\
 &=&\frac{r^{1-a}}{C_r} \partial_t v(x_0+r x,r y,t_0+r^{2s}t)\\
 &=& \partial_{t}v_r(x,y,t).\\
\end{eqnarray*}
Hence, we have proved that
\begin{equation}\label{eqnvk}
\begin{cases}
L_a v_r=\frac{r^{2}}{C_r}y^{a} g(x_0+r x) &\mbox{in } \mathbb{R}^{n+1}_{+}\times[0,T]\setminus\{v_r(x,0,t)=0\},\\
v_{r}\ge 0,& \mbox{on } \{y=0\},\\
\lim_{y\downarrow 0} y^{a}\partial_{y}v_r=\partial_{t}{v_r} & \mbox{on } \{v_r(x,0,t)>0\},\\
v(x,0,0)=0.
\end{cases}
\end{equation}
Reducing the size of $\eta$ if needed and taking $\delta$ sufficiently small,
we can take the partial derivative $\partial_{e'}$ (resp. $\kappa\partial_{e'}-\partial_{t}$) in \eqref{eqnvk},
and using \eqref{g}, \eqref{largeatmanypoints}, \eqref{lower}, and Lemma \ref{lem-max-princ}, we deduce that
\begin{equation}\label{this}
\partial_{e'}v_r\ge c_2 y^{2s}  \qquad  (\text{resp. }\quad \kappa\,\partial_{e'}v_r-\partial_t v_r\ge c_2 y^{2s}) \qquad \mbox{in }\Q_{1/2}
\end{equation}
provided $\delta$ is sufficiently small.
In particular, this proves  \eqref{5-5} and the last inequality in \eqref{5-3}.
Moreover, using that $v_r$ is a rescaling of $v$, \eqref{this} implies that
$$ \partial_{e'} v \ge 0 \quad \mbox{in}\quad  \Q_{r/2}(x_0,t_0),$$
so \eqref{5-1} follows (up to replace $r$ by $r/2$).
\vspace{3mm}

We next prove \eqref{5-2}.
For that, let $\theta\in \mathbb{S}^{n-1}$.
Since by Proposition \ref{proprescalings}
$$|2\partial_{e}v_r-\partial_{\theta}v_r-2\partial_{e}u_0-\partial_{\theta}u_0|\leq 3 \delta,$$
applying as before  the Lemma \ref{lem-max-princ} to $2\partial_{e}v_r-\partial_{\theta}v_r$, we conclude that
\[2\partial_{e}v_r\geq \partial_{\theta}v_r\qquad \mbox{for any direction $\theta\in \mathbb{S}^{n-1}$},\]
therefore
\begin{equation}\label{5-2-1}
2\partial_{e}v_r\geq |\nabla_x v_r| \qquad \mbox{in }\Q_{1/2}.
\end{equation}
On the other hand, since we also have
$$|4\partial_{e'}v_r-\partial_{e}v_r-4\partial_{e'}u_0-\partial_{e}u_0|\leq 5 \delta,$$
we can also apply Lemma \ref{lem-max-princ} to $4\partial_{e'}v_r-\partial_{e}v_r$ for any vector $e' \in \mathbb S^{n-1}$ with $e'\cdot e\geq 1/2$ to get
\begin{equation}\label{5-2-2}
4\partial_{e'}v_r\geq \partial_{e}v_r \qquad \mbox{in }\Q_{1/2}.
\end{equation}
Hence, it follows by \eqref{5-2-1} and \eqref{5-2-2} that
\[\partial_{e'}v_r\geq \frac18|\nabla_x v_r|, \quad  \mbox{for any}\quad e'\cdot e\geq 1/2,\]
which yields \eqref{5-2} with $r/2$ in place of $r$.
\vspace{3mm}

Finally, we prove the first inequality in \eqref{5-3}.
For this we simply notice that, since $\partial_t v_r>0$ in $\{v_r>0\}$ (by the strong maximum principle), there exists $c>0$ such that
$$
\partial_t v_r\geq c>0\qquad \text{in }\Q_1\cap \{x\cdot e\geq c_{n,a}\},
$$ where $c_{n,a}$ is defined in Lemma \ref{lem-max-princ}. Thus
$$
\partial_t v_r-\gamma\, \partial_{e}v_r\geq c/2 \qquad \text{in }\Q_1\cap \{x\cdot e\geq c_{n,a}\},
$$
provided that $\gamma>0$ is small enough, and we conclude that
$\partial_t v_r\geq \gamma \partial_{e}v_r$ in $\Q_{1/2}$ as before.
\end{proof}

To finish this section, we prove higher regularity in time for the solution $u$ at any regular point.

\begin{prop}\label{optimal_t}
Let $\varphi$ be an obstacle satisfying \eqref{obstacle}, let $u$ be the solution of \eqref{obst-pb} with $s\in(\frac12,1),$ and let $(x_0,t_0)$ be a regular free boundary point with exponent $\epsilon>0$ and modulus $\nu$. Then
\[\|\partial_t u\|_{{C^s_{x,t}}(Q_r(x_0,t_0))}+\|\nabla u\|_{C^s_{x,t}(Q_r(x_0,t_0))}\leq C,\]
where $C$ and $r>0$ depend only on $n$, $s$, $\epsilon$, and $\nu$.
\end{prop}

\begin{proof}
Let $v=u-\varphi$.
By the results of \cite{CF}, we know that
\[\|\nabla u\|_{C^s(\R^n)}\leq C.\]
Notice that, since $\varphi$ is independent of $t$, it is enough to prove the desired regularity of $v$.
For that purpose, note that by Corollary \ref{cor-Lip} the free boundary is Lipschitz in $x$ and $t$.
Hence, by \eqref{5-3} and the optimal $C^{1+s}_x$ regularity of solutions in space established in \cite{CF} we get that
\begin{equation}\label{simpli}
0<\partial_tv<C\,\partial_{e}v\leq C\,d_x^{s}\leq C \,d_p^{s} \qquad \text{in } Q_r(x_0,t_0),
\end{equation}
where $d_x(x,t):=\textrm{dist}(x,\{v(\cdot,t)=0\})$ denotes the Euclidean distance in $\R^n\times \{t\}$ to the free boundary, and $d_p$ the parabolic one in $\R^n\times \R$.
\vspace{3mm}

Let $(\bar x,\bar t)$ be any point in $\{v>0\}\cap Q_r(x_0,t_0)$, set $R:=d_p(\bar x,\bar t)/2>0$, and
define
$$w(x,t):=\partial_tv(x_0+R(x-x_0), t_0+R^{2s}(t-t_0)).$$
Fix $e\in \mathbb{S}^{n-1}$.
By \eqref{simpli} and interior regularity estimates for the fractional heat equation (see for example \cite[Theorem 1.3]{XX} or \cite[Theorem 2.2]{Serra}), it follows that
$$
\sup _{t\in[-1/2,0]} [w]_{{C}_x^{1}(B_{1/2})}\leq C\,R^{s}\quad \textrm{and}\quad
\sup _{t\in[-1/2,0]} [w]_{{C}_x^{s}(B_{1/2})}\leq C\,R^{s}.
$$
Therefore the previous inequalities imply that
$$\sup_{t\in(t_0-\frac{R^{2s}}{2},t_0]}\|\nabla \partial_tv\|_{L^{\infty}(B_{R/2}(x_0))}\leq C d_p(\bar x,\bar t)^{s-1},\quad
\sup_{t\in(t_0-\frac{R^{2s}}{2},t_0]}\|\partial_tv\|_{C^s_x(B_{R/2}(x_0))}\leq C.
$$
Since this can be done for any $(\bar x,\bar t)\in \{v>0\}\cap Q_r(x_0,t_0)$, and using again that (thanks to the Lipschitz regularity of the free boundary) $d_x$ and $d_p$ are comparable, we deduce that
\begin{equation}\label{imp_t}
|\nabla \partial_tv|\leq C_1 d_x^{s-1}\quad \textrm{in}\quad Q_r(x_0,t_0),\quad \textrm{and}\quad \|\partial_tv\|_{C^s_x(Q_r(x_0,t_0))}\leq C.
\end{equation}
Now, by \eqref{simpli} and \eqref{imp_t} we have that, for any $e\in \mathbb{S}^{n-1}$,
$$|(\partial_e v)^{1-s}(\partial_{te} v)^{s}|\leq C.$$
The previous inequality implies that
$$|\partial_{t} (\partial_e v)^{\frac{1}{s}}|\leq C,$$
that is, $(\partial_e v)^{\frac{1}{s}}\in {\rm Lip}_t$, which yields in particular that
\begin{equation}\label{blz}
\|\nabla v\|_{C^s_{x,t}(Q_r(x_0,t_0))}\leq C.
\end{equation}
Recalling that $\partial_tv$ and $\nabla v$ vanish on the contact set,
the previous inequality combined with \eqref{5-3} implies that
\begin{equation}\label{**}
|\partial_tv(x_1,t_1+\tau)-\partial_tv(x_1,t_1)|\leq C_0|\tau|^s
\end{equation}
for all points $(x_1,t_1)$ in $\{v=0\}\cap Q_{r/2}(x_0,t_0)$ and any $\tau \in (0,r/2)$.
\vspace{3mm}

We now prove that \eqref{**} yields $\partial_tv\in C^s_t(Q_{r/8}(x_0,t_0))$.
First, recall that $\partial_{tt}v \geq -\hat C$ by Lemma \ref{semiconvex}.
Hence, if $\psi \in C^\infty_c(Q_{2r}(x_0,t_0))$ is a nonnegative function with $\psi\equiv 1$ in $Q_{r}(x_0,t_0)$, we have
\begin{multline*}
\int_{Q_{r}(x_0,t_0)} (\partial_{tt}v+\hat C)\,dx \,dt \leq \int_{Q_{r}(x_0,t_0)} (\partial_{tt}v+\hat C)\psi \,dx \,dt\\
=\int_{Q_{2r}(x_0,t_0)} \bigl(v\,\partial_{tt}\psi+\hat C\,\psi\bigr) \,dx \,dt \leq C.
\end{multline*}
In particular, this implies that the function
\[w(x,t):=\frac{\partial_tv(x,t+\tau)-\partial_tv(x,t)}{\tau^s}\]
belongs to $L^1(Q_{r/2}(x_0,t_0))$ with a bound independent of $\tau \in (0,r/2)$.

Since  $w$ solves the fractional heat equation in the set $\{v>0\}$, and it is bounded by $C_0$ on $\{v=0\}\cap Q_{r/2}(x_0,t_0)$ by \eqref{**},
the function $\tilde w:=\max(w,C_0)$ is a subsolution in $Q_{r/2}(x_0,t_0)$ which belongs to $L^1(Q_{r}(x_0,t_0))$.
Considering a cut-off function $\psi \in C^\infty_c(B_r(x_0))$ with $\psi \equiv 1$ in $B_{3r/8(x_0)}$,
we see that $\hat w:=\tilde w \psi$ solves
$$
\partial_t \hat w +(-\Delta)^s\hat w \leq -(-\Delta)^s[(1-\psi) \tilde w] \qquad \text{in }Q_{r/4}(x_0,t_0).
$$
Since $(-\Delta)^s[(1-\psi) \tilde w]$ is universally bounded inside $Q_{r/4}(x_0,t_0)$,
we can apply \cite[Corollary 6.2]{ChD} to deduce that $\tilde w\in L^\infty(Q_{r/8}(x_0,t_0))$.
This proves that
\[\frac{\partial_tv(x,t+\tau)-\partial_tv(x,t)}{\tau^s}\leq C \qquad \text{in }Q_{r/8}(x_0,t_0)\qquad \forall\, \tau \in (0,r/2),\]
which implies that $\partial_tv\in C^s_t(Q_{r/8}(x_0,t_0))$, as desired.
\end{proof}

\section{$C^{1,\alpha}$ regularity of the free boundary in $x$}
\label{sec6}

We prove now that the free boundary is $C^{1,\alpha}$ in $x$ near regular points.
For this, we need some steps: first, we show that the set of regular points is open; then, by the results of the previous section, we deduce that the regular set is $C^1_x$; finally, by using the results in \cite{RS-C1}, we conclude the $C^{1,\alpha}_x$ regularity of the free boundary.

We will need the following result (see \cite[Lemma 4.1]{RS-C1}) which states the existence of a positive subsolution of homogeneity $s+\gamma$ vanishing outside of a convex cone that is very close to a half space.

\begin{lem}\label{homog-subsol}
Let $s\in(0,1)$, and $e\in \mathbb{S}^{n-1}$.
For every $\gamma\in(0,s)$ there is $\eta>0$ such that the function
\[ \Phi(x) := \left( e\cdot x- \frac{\eta}{4} |x| \left(1- \frac{(e\cdot x)^2}{|x|^2} \right)\right)_+^{s+\gamma}\]
satisfies
\[\begin{cases}
(-\Delta)^s \Phi \leq -c_\gamma\,d^{\gamma-s}<0 \quad & \mbox{in }\mathcal C_\eta  \\
\Phi = 0  \quad & \mbox{in }\R^n \setminus \mathcal C_\eta,\\
\end{cases}\]
where
\[\mathcal C_\eta: = \left\{ x \in \R^n\ : \ e\cdot\frac{x}{|x|} >  \frac{\eta}{4} |x| \left(1- \frac{(e\cdot x)^2}{|x|^2} \right) \right\},\qquad
 d(x):={\rm dist}(x,\R^n\setminus \mathcal C_\eta).\]
Here the constants $c_\gamma$ and $\eta$ depend only on $\gamma$ and $s$.
\end{lem}

Using the previous Lemma, we now show that if $(x_0,t_0)$ is a regular free boundary point, then all free boundary points in a neighborhood of $(x_0,t_0)$ are also regular.

\begin{prop}\label{contagi}
Assume that $(x_0,t_0)$ is a regular free boundary point with exponent $\epsilon>0$ and modulus~$\nu$.
Set $\gamma_s:=\min\left\{2s-1,1-s\right\}$.
For any $\gamma \in (0,\gamma_s)$ there are $r>0$ and $c>0$ such that, for every free boundary point $(x_1,t_1)\in \partial \{u=\varphi\}\cap Q_r(x_0,t_0)$, we have
\begin{equation}\label{nondegeneracy}
u(x_1+\lambda e,t_1)\geq c\lambda^{1+s+\gamma},\qquad \partial_e u(x_1+\lambda e,t_1)\geq c\lambda^{s+\gamma}.
\end{equation}
In particular, every point on $\partial \{u=\varphi\}\cap Q_r(x_0,t_0)$ is regular with exponent $\frac{1-s-\gamma}{2}>0$ and modulus of continuity $\tilde\nu(\rho):=c\rho^{(s+\gamma-1)/2}$.
\end{prop}

\begin{proof}
Fix $0<\gamma<\gamma_s$, and let $\eta>0$ and $\mathcal C_\eta$ be given by Lemma~\ref{homog-subsol} (note that $\gamma_s<s$).

Let $v$ be given by \eqref{v}, and $v_r$ be defined  as in \eqref{xr}.
Also, let $\kappa>0$ be a small number to be fixed later. By Proposition \ref{free-bdry-Lipschitz} and Corollary \ref{cor-Lip}, there exists $e\in \mathbb{S}^{n-1}$
and $r>0$ small enough such that \eqref{5-3} holds and
\[(x_1+\mathcal C_{\eta/4})\cap B_2(x_1)\subset\{v_r(\cdot,t)>0\} \qquad \forall\,x_1 \in \{v_r(\cdot,t)>0\} \cap B_{1/4}, \,\forall\,t\in(-1,1).\]
Noticing that the function $v_r$ solves
\[\begin{cases}
L_a v_r=y^a g(x)  \quad &\mbox{in } \Q_1\cap\{y>0\}\\
\lim_{y\downarrow0}y^{1-2s}\partial_yv_r=\partial_t v_r \quad &\mbox{on } \Q_1\cap \{y=0\}\cap \{v_r>0\}.
\end{cases}\]
with
\[g(x):=\frac{r^2(\Delta\varphi)(x_0+rx)-r^2\Delta \varphi(x_0)}{\|v\|_{L^\infty(\Q_r(x_0,t_0))}},\qquad |\nabla g|\leq C_1r^{1+\epsilon},\]
as in \eqref{imp_t} it follows by \eqref{5-3} combined with interior estimates that
$$|\nabla \partial_t v_r|\leq C\kappa d_x^{s-1}$$
for some $C$ independent of $\kappa$.

Now, fix $t\in (-1,1)$ and define
\[w(x,y):=\partial_e v_r(x,y,t).\]
Thanks to the previous considerations, choosing $r$ small enough we have
\[\begin{cases}
|L_a w|\leq \kappa y^a  \quad &\mbox{in } \B_2\cap\{y>0\}\\
|\lim_{y\downarrow0}y^{1-2s}\partial_yw|\leq C\kappa d_x^{s-1} \quad &\mbox{on } \B_2\cap \{y=0\}\cap \{w>0\}.
\end{cases}\]
Moreover, by \eqref{5-3} and \eqref{5-4} we have
\begin{equation}\label{c1-1}
w\geq c_1>0\quad\textrm{in the set}\quad \B_2\cap \{x\cdot e\geq \eta/16\}.
\end{equation}
and
\begin{equation}\label{c1-2}
w\geq c_1y^{2s}\quad \mbox{in}\quad \B_2.
\end{equation}
We want to use the function $\Phi$ in Lemma \ref{homog-subsol} as a subsolution at any free boundary point of $w$ near $0$. To this aim we note that,
as a consequence of \eqref{c1-1}, if $x_1$ is a free boundary point close to $0$ then
\begin{equation}\label{ayuda}
 x_1\cdot e\leq \eta/16.
\end{equation}
Denote $\Phi(x,y)$ the extension of $\Phi(x)$ in $\R^{n+1}_+$, which satisfies
\[\begin{cases}
L_a \Phi=0  \quad &\mbox{in } \{y>0\}\\
\lim_{y\downarrow0}y^{1-2s}\partial_y\Phi\geq c_\gamma\, d_x^{\gamma-s} \quad &\mbox{on } \{y=0\}\cap \mathcal C_\eta\\
\Phi=0 \quad &\mbox{on } \{y=0\}\setminus\mathcal C_\eta.
\end{cases}\]
We recall that $\Phi$ can be written via the Poisson formula as 
\begin{equation}\label{poisson}
\Phi(x,y)=C_{n,s}\, y^{2s} \int_{\R^n} \frac{\Phi(z,0)}{(|x-z|^2+y^2)^{(n+2s)/2}}\,dz\qquad \forall\,y>0
\end{equation}
(see \cite[Section 2.4]{CSext}).

Consider now $x_1\in\partial\{w=0\}\cap B_{1/4}$, and define the function
\[\psi(x,y):=c_2\Phi(x-x_1,y)+\frac{\kappa}{4(1-s)} y^2,\]
so that
\begin{equation}\label{max1}
L_a\psi = \kappa y^a\geq L_aw  \qquad \mbox{in } \B_1(x_1)\cap \{0<y<\eta\}.
\end{equation}
Recalling that $(x_1+\mathcal C_{\eta})\cap \B_1(x_1)\subset\{w>0\}$
(see \eqref{c1-1}), we have
\begin{multline}\label{max2}
\lim_{y\downarrow0}y^{1-2s}\partial_y\psi =-(-\Delta)^s \Phi\geq c_2 c_\gamma\, d_x^{\gamma-s}\\
\geq \kappa d_x^{s-1}\geq \lim_{y\downarrow0}y^{1-2s}\partial_y w
   \quad \mbox{on } \B_1(x_1)\cap \{y=0\}\cap (x_1+\mathcal C_\eta),
\end{multline}
provided that $\kappa>0$ is small enough. Also
\begin{equation}\label{max3}
\psi=0\leq w\qquad \mbox{on}\quad \bigl(\B_1(x_1)\cap \{y=0\} \bigr)\setminus (x_1+\mathcal C_\eta),
   \end{equation}
and it follows by \eqref{c1-2} that
  \begin{equation}\label{max4}
\psi\leq w\qquad \mbox{on}\quad \B_1(x_1)\cap \{y=\eta\}
   \end{equation}
provided $c_2$ and $\kappa$ are sufficiently small.

We now check what happens on $\partial\B_1(x_1)\cap \{0<y<\eta\}$.
First of all we see that, thanks to \eqref{c1-1},
  \begin{equation}\label{max5}
\psi\leq w\qquad \mbox{on}\quad \partial\B_1(x_1)\cap \{x\cdot e>\eta/16\}\cap \{0<y<\eta\}
   \end{equation}
   provided $c_2$ and $\kappa$ are sufficiently small.
   Finally, since $\Phi$ vanishes on a uniform neighborhood $N\subset \R^n$ of $\partial\B_1(x_1)\cap \{y=0\}\cap\{x\cdot e\leq \eta/16\}$,
   it follows by \eqref{poisson} that $\Psi \leq C \,y^{2s}$ on $N\times [0,\eta]$.
   This implies that $\psi\leq C(c_2+\kappa)y^{2s}$ on $N\times [0,\eta]$,
   that combined with \eqref{c1-2} proves that
     \begin{equation}\label{max6}
\psi\leq w\qquad \mbox{on}\quad \partial\B_1(x_1)\cap \{x\cdot e\leq \eta/16\}\cap \{0<y<\eta\}
   \end{equation}
   if $c_2$ and $\kappa$ are sufficiently small.
Hence, combining \eqref{max1}, \eqref{max2}, \eqref{max3},
\eqref{max4}, \eqref{max5}, and \eqref{max6}, it follows by the maximum principle that
$w\geq \psi$ in $\B_1(x_1)$.
In particular we deduce that
\[w(x,0) \ge \psi(x,0)=c_2\Phi(x-x_1)\quad \textrm{in}\quad B_1\cap\{y=0\}.\]
and recalling that $\Phi$ is homogeneous of degree $s+\gamma$, we find
\[\partial_e v_r(x_1+\lambda e,0,t)=w(x_1+\lambda e,0)>c_3\lambda^{s+\gamma}\]
for $\lambda\in (0,1)$.
Integrating in $\lambda$, we get $v_r(x_1+\lambda e,0,t)\geq c_4\lambda^{1+s+\gamma}$ for $\lambda\in(0,1)$.
Since $1+s+\gamma<2$, this means that $(x_1,t)$ is a regular free boundary point for $v_r$, with $\epsilon=\frac12(1-s-\gamma)$ and $\tilde\nu(\rho)=c_4\rho^{\frac12(s+\gamma-1)}$. Since $(x_1,t)$ was arbitrary in $B_{1/4}\times (-1,1)$ and
$v_r$ is a rescaled version of $u-\varphi$, the Proposition follows.
\end{proof}

Using the previous result, we find the following.

\begin{cor}\label{cor-C1x}
Assume $(x_0,t_0)$ is a regular point.
Then, there is $r>0$ such that the free boundary is $C^1_x$ in $Q_r(x_0,t_0)$, with a uniform modulus of continuity.
\end{cor}

\begin{proof}
By Proposition \ref{contagi}, all free boundary points in $Q_r(x_0,t_0)$ are regular, with a uniform exponent $\epsilon>0$ and modulus $\nu$.
Thus, by Corollary \ref{cor-Lip} the free boundary is $C^1_x$ at all such points, with a uniform modulus of continuity.
\end{proof}

Finally, using the results of \cite{RS-C1}, we deduce that the set of regular free boundary points is $C^{1,\alpha}$ in $x$.

\begin{cor}\label{cor-C1alphax}
Assume $(x_0,t_0)$ is a regular point.
Then, there is $r>0$ such that the free boundary is $C^{1,\alpha}_x$ in $Q_r(x_0,t_0)$, for some small $\alpha>0$.

Furthermore, there exists $c>0$ such that, for every free boundary point $(x_1,t_1)\in Q_r(x_0,t_0)\cap \partial\{u=\varphi\}$, we have
\[(u-\varphi)(x,t_1)=c(x_1,t_1)d_x^{1+s}(x,t_1)+o(|x-x_1|^{1+s+\alpha}),\]
where $c(x_1,t_1)\geq c$ and $d_x(x,t)={\rm dist}(x,\{u(\cdot,t)=\varphi\})$.
\end{cor}

\begin{proof}
Let $(x_1,t_1)\in Q_r(x_0,t_0)\cap \partial\{u=\varphi\}$ be any free boundary point, and set
$$w(x):=u(x+x_1,t_1)-\varphi(x+x_1).$$
Also, denote $\Omega:=\{w>0\}$ and
recall that, by Corollary \ref{cor-C1x}, $\Omega$ is $C^1$ in a neighborhood of the origin.
After a rotation, we may assume that the normal vector to $\partial\Omega$ at the origin is $e_n$.
Recall also that $0$ is a regular free boundary point with an exponent $\epsilon>0$
and a modulus $\nu$ which are independent of the point  $(x_1,t_1)\in Q_r(x_0,t_0)\cap \partial\{u=\varphi\}$.
Throughout the proof, $C$ and $c$ will denote positive constants independent of $x_1$ and $t_1$.
\vspace{3mm}

First, we rescale the function $w$ as
\[w_k(x)=\frac{w(r_kx)}{r_k\|\nabla w\|_{L^\infty(B_{r_k})}},\qquad \nabla w_k(x)=\frac{(\nabla w)(r_kx)}{\|\nabla w\|_{L^\infty(B_{r_k})}},\]
along a sequence $r_k\to0$ such that $\|\nabla w\|_{L^\infty(B_{r_k})}\geq c(r_k)^{1-\epsilon}$, $\|\nabla w_k\|_{L^\infty(B_1)}=1$,  and
\[|\nabla w_k(x)|\leq C(1+|x|^{1-\epsilon})\]
(compare with Lemma \ref{lem-rescalings}).
On the other hand, recalling that
\[(-\Delta)^sw=-(-\Delta)^s\varphi-\partial_t u,\]
it follows by \eqref{obstacle} and  \eqref{imp_t} that
\[|(-\Delta)^s\partial_e w|\leq C(1+d_x^{s-1})\quad \textrm{in}\quad \Omega\cap B_r\]
for all $e\in \mathbb{S}^{n-1}$.
This implies that the rescaled functions $w_k$ satisfy
\[|(-\Delta)^s\partial_e w_k|\leq \kappa(1+d_x^{s-1})\quad \textrm{in}\quad \Omega\cap B_{r/r_k}\]
for $k$ large enough,
with $\kappa>0$ as small as desired.
\vspace{3mm}

Consider $e\in \mathbb S^{n-1}$ with $e\cdot e_n\geq 1/2$. Then,
it follows by \eqref{5-1} and \eqref{5-2}  that, for $k$ large enough, $\partial_e w_k\geq0$ in $B_{1/\kappa}$ and  $\sup_{B_1}\partial_e w_k\geq c>0$.
This allows us to apply \cite[Theorem 1.3]{RS-C1} (see also Remark 5.5 therein) and deduce that
\[\left\| \frac{\partial_e w_k}{\partial_{e_n} w_k}\right\|_{C^\alpha(\overline\Omega \cap B_1)}\leq C\]
for all such $e\in \mathbb{S}^{n-1}$.
In particular, setting $e=(e_i+e_n)/\sqrt2$, $i=1,...,n-1$, and using that $w_k$ is a rescaled version of $w$, choosing $k$ large enough but fixed,
the previous inequality yields
\begin{equation}\label{quotient}
\left\| \frac{\partial_{e_i} w}{\partial_{e_n} w}\right\|_{C^\alpha(\overline\Omega \cap B_r)}\leq C,
\end{equation}
for some $r>0$ small.

Now, notice that the normal vector $\hat \nu(x)$ to the level set $\{w=\lambda\}$ for $\lambda>0$ can be written as
\[\hat \nu(x)=\frac{\nabla w}{|\nabla w|}(x)=(\hat \nu^1(x),...,\hat \nu^n(x)),\]
\[\hat \nu^i(x)=\frac{\partial_{e_i}w/\partial_{e_n}w}{\sqrt{\sum_{j=1}^n(\partial_{e_j}w/\partial_{e_n}w)^2}}.\]
Hence, \eqref{quotient} implies that $|\hat \nu(x)-\hat \nu(y)|\leq C|x-y|^\alpha$ whenever $x,z\in \{w=\lambda\}\cap B_r$, with $C$ independent of $\lambda>0$.
Letting $\lambda\to0$, we find that $\partial\{w=0\}\cap B_r$ is $C^{1,\alpha}$.
\vspace{3mm}

Finally, once we know that $\Omega$ is of class $C^{1,\alpha}$, we can apply in \cite[Theorem 1.2]{RS-C1} (see also Remark 3.4 therein) to deduce that find
\[\left\| \partial_e w/d_x^s\right\|_{C^\alpha(\overline\Omega \cap B_r)}\leq C\]
for some $\alpha>0$, which yields
\[u(x,t_1)-\varphi(x)=c(x_1,t_1)d_x^{1+s}(x,t_1)+o(|x-x_1|^{1+s+\alpha}).\]
Finally, by \eqref{nondegeneracy} in Proposition \ref{contagi}, we deduce $c(x_1,t_1)\geq c>0$.
\end{proof}

\section{$C^{1,\beta}$ regularity of the free boundary in $x$ and $t$}
\label{sec7}

We finally prove that the free boundary is $C^{1,\beta}$ both in $x$ and $t$

\begin{prop}\label{C1alpha}
Let $\varphi\in C^4(\R^n)$ be an obstacle satisfying \eqref{obstacle}, and $u$ be the solution to \eqref{obst-pb}, with $s\in(\frac12,1)$.
Assume that $(x_0,t_0)$ is a regular free boundary point with exponent~$\epsilon>0$ and modulus~$\nu$,
and let $G:\R^{n-1}\times \R\to \R$ be as in Corollary \ref{cor-Lip}.
Then $G$ is of class $C^{1,\beta}_{x',t}$ inside $Q_{r}(x_0,t_0)$ for some small $\beta>0$.
\end{prop}

\begin{proof}
Let $v=u-\varphi$.

For every free boundary point $(x_1,t_1)\in Q_r(x_0,t_0)\cap \partial\{u=\varphi\}$, by Corollary \ref{cor-C1alphax} we have the expansion
\begin{equation}\label{expansion}
\left|v(x,t_1)-c(x_1,t_1)d_x^{1+s}(x,t_1)\right|\leq C|x-x_1|^{1+s+\alpha}\qquad \textrm{for}\quad t=t_1,
\end{equation}
where $c(x_1,t_1)\geq c>0$, and $d_x(x,t)={\rm dist}(x,\{u(\cdot,t)=\varphi\})$.

If $\hat\nu(x',t) \in \mathbb S^{n-1}$ is the normal vector (in $x$) to the free boundary at $(x',G(x',t),t)$, then denoting by $c(x',t):=c(x', G(x',t), t)$, and using that, by Corollary \ref{cor-C1alphax}, the function $x'\mapsto G(x',t)$ is of class $C^{1,\alpha}$ inside $Q_r(x_0,t_0)$, it follows by \eqref{expansion} that
\begin{equation}\label{expansion2}
\left|v(x,t)-c(x',t)\bigl(\{x-(x',G(x',t))\}\cdot\hat\nu(x',t)\bigr)_+^{1+s}\right|\leq C\bigl|x-(x',G(x',t))\bigr|^{1+s+\alpha}.
\end{equation}

Our objective is to prove that the function $G(x',t)$ is of class $C^{1+\beta}$ in the $t$-variable.
For this, we first show that $c(x',t)$ is $C^\gamma_t$ for some $\gamma>0$.
\vspace{3mm}

For that purpose, let us consider two free boundary points $(x',G(x',t),t)$ and $(x',G(x',t+\tau),t+\tau)$ in $Q_r(x_0,t_0)$, with $\tau>0$ small.
We fix a number $h\in(0,1)$ with $h\gg\tau>0$ and we next compare the expansions \eqref{expansion2} at the points
\[\begin{split}
A&:=(x',G(x',t)+h,t),\\
B&:=(x',G(x',t)+h,t+\tau)
,\end{split}\]
to show that $c(x',t)$ is $C^\gamma$ in the $t$-variable.

On the one hand, by \eqref{expansion2} at the point $A$ we have
\begin{equation}\label{exp1}
\left|v(A)-c(x',t)\bigl((0,h)\cdot\hat\nu(x',t)\bigr)_+^{1+s}\right|\leq C\,h^{1+s+\alpha}.
\end{equation}
On the other hand, \eqref{expansion2} at the point $B$ gives
\begin{multline}\label{exp2}
\left|v(B)-c(x',t+\tau)\bigl((0,h+G(x',t+\tau)-G(x',t))\cdot\hat\nu(x',t+\tau)\bigr)_+^{1+s}\right|  \\
\leq C|h+G(x',t+\tau)-G(x',t)|^{1+s+\alpha}.
\end{multline}
Now, since $G(x',t)$ is Lipschitz in $t$ (by Corollary \ref{cor-Lip}), and $\hat\nu(x',t)$ is $C^\alpha$ in $x'$, we find
\begin{equation}\label{normalvector}
|\hat\nu(x',t)-\hat\nu(x',t+\tau)|\leq C\,\tau^{\alpha/2}.
\end{equation}
Indeed, by $C^{1+\alpha}_x$ regularity of $G$ we have
\[\left| G(\bar x',t)-G(x',t)-\nabla_{x'} G(x',t)\cdot (\bar x'-x')\right|\leq C|\bar x'-x'|^{1+\alpha}.\]
Combining this estimate with the same one at time $t+\tau$, and using that
\begin{equation}\label{g-Lip-t}
|G(x',t+\tau)-G(x',t)|\leq C\,\tau
\end{equation}
(since $G$ is Lipschitz in $t$),
we find
\[(\bar x'-x')\cdot\bigl(\nabla_{x'} G(x',t)-\nabla_{x'} G(x',t+\tau)\bigr)\leq C\,\tau+C|\bar x'-x'|^{1+\alpha}.\]
Choosing $\bar x'$ such that $|\bar x'-x'|=\tau^{1/2}$ and $\bar x'-x'$ points in the same direction as $\nabla_{x'} G(x',t)-\nabla_{x'} G(x',t+\tau)$, we deduce that
\begin{equation}\label{poma}
\bigl|\nabla_{x'} G(x',t)-\nabla_{x'} G(x',t+\tau)\bigr|\leq C\,\tau^{1/2}+C\,\tau^{\alpha/2}\leq C\,\tau^{\alpha/2}.
\end{equation}
Thus, since $\hat \nu(x',t)$ is the normal vector to the graph of $G$, \eqref{normalvector} follows.

Using \eqref{normalvector} and \eqref{g-Lip-t}, and recalling that $h\gg\tau$, it follows from \eqref{exp2} that
\[\left|v(B)-c(x',t+\tau)\bigl((0,h)\cdot\hat\nu(x',t)\bigr)_+^{1+s}\right|\leq C\,h^{1+s+\alpha/2}+C\,\tau^{1+s}.\]
Combining the previous inequality with \eqref{exp1}, we find
$$\bigl|v(B)-v(A)\bigr|\geq \bigl(c(x',t+\tau)-c(x',t)\bigr)\bigl((0,h)\cdot\hat\nu(x',t)\bigr)_+^{1+s}-C\,h^{1+s+\alpha/2}-C\,\tau^{1+s}$$
Now, if $r$ is small enough, then by Corollary \ref{cor-Lip} we know that $\hat\nu(x',t)\cdot e_n\geq\frac12$, therefore
\begin{equation}\label{orden}
\bigl|v(B)-v(A)\bigr|\geq \frac12\bigl|c(x',t+\tau)-c(x',t)\bigr|\,h^{1+s}-C\,h^{1+s+\alpha/2}-C\,\tau^{1+s}.
\end{equation}
Using that $v$ is Lipschitz in $t$, this yields
\[|v(A)-v(B)|\leq C\,\tau,\]
so, by \eqref{orden},
\[\bigl|c(x',t+\tau)-c(x',t)\bigr|\,h^{1+s}\leq C\,h^{1+s+\alpha/2}+C\,\tau.\]
Thus, choosing $h=\sqrt{\tau}$, we deduce
\[\bigl|c(x',t+\tau)-c(x',t)\bigr|\leq C\,h^{\alpha/2}+C\,h^{2-1-s}\leq C\,h^{\alpha/2}=C\,\tau^{{\alpha}/{4}},\]
provided that $\alpha\in(0,1-s)$. Hence, this proves that $c(x',t)$ in \eqref{exp1} is $C^\gamma$ in~$t$ with $\gamma=\alpha/4$.
\vspace{3mm}

Using this, we now show that $\partial_t G(x',t)$ is $C^{\beta}$ in $t$.
We compare the expansions \eqref{expansion2} at the three points
\[\begin{split}
A&=(x',G(x',t)+h,t)\\
B&=(x',G(x',t)+h,t+\tau)\\
C&=(x',G(x',t)+h,t-\tau)
,\end{split}\]
with $h\gg \tau$.

As before, by \eqref{exp2} and \eqref{normalvector} we have
\[\left|v(B)- c(x',t+\tau)\bigl\{(0,h+G(x',t+\tau)-G(x',t))\cdot\hat\nu(x',t)\bigr\}_+^{1+s}\right| \leq C\,h^{1+s+\alpha/2}.\]
Now, using that $c(x',t)$ is $C^{\alpha/4}_t$, this yields
\[\left|v(B)- c(x',t)\bigl(h+G(x',t+\tau)-G(x',t)\bigr)^{1+s}\bigl(e_n\cdot\hat\nu(x',t)\bigr)_+^{1+s}\right| \leq C\,h^{1+s+\alpha/4}.\]
Analogously, we have
\[\left|v(C)- c(x',t)\bigl(h+G(x',t-\tau)-G(x',t)\bigr)^{1+s}\bigl(e_n\cdot\hat\nu(x',t)\bigr)_+^{1+s}\right| \leq C\,h^{1+s+\alpha/4}.\]
Therefore, combining the previous two inequalities with \eqref{exp1}, we find
\begin{multline*}
\bigl|v(B)+v(C)-2v(A)\bigr|\\
\geq  c(x',t)\left|h^{1+s}\left((1+a)^{1+s}+(1+b)^{1+s}-2\right) \right|\bigl(e_n\cdot \hat\nu(x',t)\bigr)^{1+s}_+-C\,h^{1+s+\alpha/4},
\end{multline*}
where
\[a:=\frac{G(x',t+\tau)-G(x',t)}{h},\qquad b:=\frac{G(x',t-\tau)-G(x',t)}{h}.\]
Recalling that $G$ is Lipschitz in $t$ (see Corollary \ref{cor-Lip}), since $h \gg \tau$ we observe that
$|a|\ll1$ and $|b|\ll1$.
Also, since $\hat\nu(x',t)\cdot e_n\geq\frac12$ and $c(x',t)\geq c>0$, 
\begin{equation}\label{orden2}
\bigl|v(B)+v(C)-2v(A)\bigr| \geq  c\left|h^{1+s}\left((1+a)^{1+s}+(1+b)^{1+s}-2\right) \right|-C\,h^{1+s+\alpha/4}.
\end{equation}
Hence, since
\[\left|(1+a)^{1+s}+(1+b)^{1+s}-2\right|\geq c|a+b|\qquad\textrm{for}\quad a,b \quad \textrm{small},\]
 it follows from \eqref{orden2} that
\[\bigl|v(B)+v(C)-2v(A)\bigr|\geq  c\,h^s\bigl|G(x',t+\tau)+G(x',t-\tau)-2G(x',t)\bigr| - C\,h^{1+s+\alpha/4}.\]
Finally, using that $u\in C^{1+s}_t$ (by Proposition \ref{optimal_t}), we get that
\[\bigl|v(B)+v(C)-2v(A)\bigr|\leq  C\,\tau^{1+s},\]
therefore
\[\bigl|G(x',t+\tau)+G(x',t-\tau)-2G(x',t)\bigr|\leq C\,h^{1+\alpha/4}+C\,\tau^{1+s}h^{-s}.\]
Setting $h=\tau^{\frac{1+s}{1+s+\alpha/4}}$, this gives
\[\bigl|G(x',t+\tau)+G(x',t-\tau)-2G(x',t)\bigr|\leq C\,\tau^{1+\beta},\]
with $\beta=\frac{\alpha s}{4+4s+\alpha}>0$.
Now, it is a standard fact that this bound implies that $G\in C^{1+\beta}_t$.
In particular, this yields $\partial_t G$ is $C^{\beta}_t$.
\vspace{3mm}

Note that, as a consequence of Corollary \ref{cor-C1alphax} and \eqref{poma} we know that $\nabla_{x'} G\in C^\beta_{x',t}$.
Although a priori we do not have informations about  the regularity of $\partial_t G$ with respect to $x'$,
we can still conclude that $G \in C^{1,\beta}_{x',t}$ (thus, in particular, $\partial_t G\in C^\beta_{x',t}$).

Indeed, fix $(x_0',t_0)$ and consider $(h',\tau) \in \R^{n-1}\times \R$ small.
Then, using that $\partial_t G\in C^{\beta}_t$ and $\nabla_{x'} G\in C^\beta_{x',t}$, we have
\begin{eqnarray*}
G(x_0'+h',t_0+\tau)&=&G(x_0',t_0+\tau)+\Bigl(\int_0^{1}\nabla_{x'}G(x_0'+sh',t_0+\tau)\,ds\Bigr)\cdot h'\\
&=&G(x_0',t_0)+\partial_tG(x_0',t_0)\,\tau +O(|\tau|^{1+\beta}) +\nabla_{x'}G(x_0',t_0)\cdot h'\\
&&+\Bigl(\int_0^{1}\nabla_{x'}G(x_0'+sh',t_0+\tau)-\nabla_{x'}G(x_0',t_0)\,ds\Bigr)\cdot h'\\
&=&G(x_0',t_0)+\partial_tG(x_0',t_0)\,\tau +O(|\tau|^{1+\beta})\\
&&+\nabla_{x'}G(x_0',t_0)\cdot h' + O\bigl((|h'|^\beta+|\tau|^\beta)|h'|\bigr)\\
&=&G(x_0',t_0)+\partial_tG(x_0',t_0)\,\tau +\nabla_{x'}G(x_0',t_0)\cdot h' \\
&&+ O(|h'|^{1+\beta}+|\tau|^{1+\beta}).
\end{eqnarray*}
This proves $G$ separates from its first order Taylor expansion by at most $|x'|^{1+\beta}+|t|^{1+\beta}$,
thus $G \in C^{1,\beta}_{x',t}$ as desired.
\end{proof}

\section{Proof of Theorem \ref{thm1}}
\label{sec8}

Let $(x_0,t_0)\in \partial\{u=\varphi\}$ be a regular free boundary point ---that is, a free boundary point at which (ii) does not hold.

By Propositions \ref{optimal_t} and \ref{C1alpha}, we have that $u\in C^{1+s}_{x,t}(Q_r(x_0,t_0))$ and the free boundary is $C^{1+\beta}_{x,t}$ in $Q_r(x_0,t_0)$, for some $\beta>0$ and $r>0$.
By Corollary \ref{cor-C1alphax}, for any free boundary point $(x_1,t_1)\in Q_r(x_0,t_0)$ we have the expansion
\[u(x,t_1)-\varphi(x)=c(x_1,t_1)d_x^{1+s}+o\bigl(|x-x_0|^{1+s+\alpha}\bigr).\]
Also, by the $C^{1,\beta}$ regularity of the free boundary in $x$ and $t$, we have
\[d_x^{1+s}=\bigl(e\cdot (x-x_0)+a(t-t_0)\bigr)_+^{1+s}+o\bigl(|x-x_0|^{1+s+\beta}+|t-t_0|^{1+s+\beta}\bigr),\]
for some $e\in \mathbb{S}^{n-1}$ and $a\in\R$.
Moreover, by monotonicity in $t$ we have $a\geq0$, and in fact, by \eqref{5-3} in Proposition \ref{free-bdry-Lipschitz}, we get $a>0$.

Combining the previous identities, and using that $(x_1,t_1)\mapsto c(x_1,t_1)$ is of class $C^{\alpha/4}$ in $x$ and $t$ (see the proof of Proposition \ref{C1alpha}), we deduce that
\[u(x,t)-\varphi(x)=c(x_0,t_0)\bigl(e\cdot (x-x_0)+a(t-t_0)\bigr)_+^{1+s}+ o\bigl(|x-x_0|^{1+s+\gamma}+|t-t_0|^{1+s+\gamma}\bigr)\]
with $\gamma:=\min\{\alpha/4,\beta\}$, $c(x_0,t_0)>0$, $a>0$, and $e\in \mathbb{S}^{n-1}$.
In particular this yields $\sup_{Q_r(x_0,t_0)}(u-\varphi)=c(x_0,t_0)r^{1+s}+o(r^{1+s+\gamma})$, and the theorem follows.
\qed

\end{document}